\documentclass[sn-mathphys,Numbered]{sn-jnl}

\usepackage{graphicx}%
\usepackage{comment}%
\usepackage{subcaption}
\usepackage{multirow}%
\usepackage{amsmath,amssymb,amsfonts}%
\usepackage{amsthm}%
\usepackage{mathrsfs}%
\usepackage[title]{appendix}%
\usepackage{xcolor}%
\usepackage{textcomp}%
\usepackage{manyfoot}%
\usepackage{booktabs}%
\usepackage{algorithm}%
\usepackage{algorithmicx}%
\usepackage{algpseudocode}%
\usepackage{listings}%
\usepackage{hyperref}
\usepackage{cleveref}
\usepackage{setspace}
\usepackage{array}
\usepackage{pdfpages}

\theoremstyle{thmstyleone}%
\newtheorem{theorem}{Theorem}%
\theoremstyle{thmstyletwo}%
\newtheorem{lemma}{Lemma}
\theoremstyle{thmstylethree}%
\newtheorem{assumption}{Assumption}
\renewcommand\proofname{\textbf{Proof}}
\newcommand{\N}{\mathcal{N}}
\newcommand{\D}{\mathcal{D}}

\begin{document}

\title[A non-monotone  trust-region method with noisy oracles and additional sampling]{A non-monotone  trust-region method with noisy oracles and additional sampling}
\author[1]{\orcid{https://orcid.org/0000-0003-3348-7233}\fnm{Nata\v sa} \sur{Kreji\'c}}\email{natasak@uns.ac.rs}
\author[1]{\orcid{https://orcid.org/0000-0001-5195-9295}\fnm{Nata\v sa} \sur{Krklec Jerinki\'c}}\email{natasa.krklec@dmi.uns.ac.rs}

\author[2]{\orcid{https://orcid.org/0000-0003-4826-1114}\fnm{\'Angeles} \sur{Mart\'{\i}nez} }\email{amartinez@units.it}
\author*[2]{\orcid{https://orcid.org/0000-0002-2937-9654}\fnm{Mahsa} \sur{Yousefi}}\email{mahsa.yousefi@phd.units.it}
\affil[1]{\orgdiv{Department of Mathematics and Informatics}, \orgname{University of Novi Sad}, \orgaddress{\street{Trg Dositeja Obradovi\' ca 4}, \city{Novi Sad}, \postcode{21000}, \country{Serbia}}}
\affil*[2]{\orgdiv{Department of Mathematics, Informatics, and Geosciences}, \orgname{University of Trieste}, \orgaddress{\street{Via Alfonso Valerio 12/1}, \city{Trieste}, \postcode{34127}, \country{Italy}}}


\abstract{In this work, we introduce a novel stochastic second-order method, within the framework of a non-monotone trust-region approach, for solving the unconstrained, nonlinear, and non-convex optimization problems arising in the training of deep neural networks. The proposed algorithm makes use of subsampling strategies that yield noisy approximations of the finite sum objective function and its gradient.
We introduce an adaptive sample size strategy based on inexpensive additional sampling to control the resulting approximation error. Depending on the estimated progress of the algorithm, this can yield sample size scenarios ranging from mini-batch to full sample functions. We provide convergence analysis for all possible scenarios and show that the proposed method achieves almost sure convergence under standard assumptions for the trust-region framework.
We report numerical experiments showing that the proposed algorithm outperforms its state-of-the-art counterpart in deep neural network training for image classification and regression tasks while requiring a significantly smaller number of gradient evaluations.}




\maketitle

\section{Introduction}\label{body1}

Deep learning (DL) as a leading technique of machine learning (ML) has attracted much attention and become one of the most popular directions of research. DL approaches have been applied to solve many large-scale problems in different fields by training deep neural networks (DNNs) over large available datasets. Let $\N=\{1, 2, \ldots, N\}$ be the index set of the training dataset $\{(x_i, y_i)\}_{i=1}^{N}$ with $N=|\N|$ sample pairs including input $x_i\in\mathbb{R}^d$ and target \textbf{$y_i\in\mathbb{R}^C$}. DL problems are often formulated as unconstrained optimization problems with an empirical risk function, in which a parametric function $\hat{h}( x_i;\cdot):\mathbb{R}^d \longrightarrow \mathbb{R}^C$ is found such that the prediction errors are minimized. More precisely, we obtain the following problem 
\begin{equation}\label{eq.problem}
    \min_{w \in \mathbb{R}^n} f(w) \triangleq \frac{1}{N} \sum_{i=1}^{N}{f_i(w)},
\end{equation}
where $w \in \mathbb{R}^n$ is the vector of trainable parameters and $f_i(w) \triangleq L(y_i, \hat{h}(x_i; w))$ with a relevant loss function $L(\cdot )$ measuring the prediction error between the target $y_i$ and the network's output $\hat{h}(x_i; w)$. The DL problem \eqref{eq.problem} is large-scale, highly nonlinear, and often non-convex, and thus it is not straightforward to apply traditional (deterministic)  optimization algorithms like steepest descent or Newton-type methods. Recently, much effort has been devoted to the development of DL optimization algorithms. Popular DL optimization methods can be divided into two \textit{general} categories, \textit{first-order} methods using gradient information, e.g. steepest gradient descent, and \textit{(higher-) second-order} methods using also curvature information, e.g. Newton methods \cite{nocedal2006numerical}. On the other hand, since the full (training) sample size $N$ in \eqref{eq.problem} is usually excessively large for a deterministic approach, these optimizers are further adapted to use subsampling strategies that aim to reduce computational costs. Subsampling strategies employ sample average approximations of the function and its gradient as follows 
\begin{equation}\label{eq.sampled-fungrad}
    f_{\N_k}(w) = \frac{1}{N_k} \sum_{i \in \N_k}{f_i(w)}, \quad \nabla f_{\N_k}(w) = \frac{1}{N_k} \sum_{i \in \N_k}{\nabla f_i(w)},
\end{equation}
where $\N_k \subseteq \N$ represents a subset of the full (training) sample set at iteration $k$ and $N_k$ is the subsample size, i.e., $N_k=|\N_k|$.

In this work, we propose a second-order trust-region (TR) algorithm \cite{conn2000trust} adapted to 
the stochastic framework where the step and the candidate point for the next iterate are obtained using subsampled function values and subsampled gradients \eqref{eq.sampled-fungrad}. The quadratic TR models are constructed by using Hessian approximations, without imposing a positive definiteness assumption, as the true Hessian in DL problems may not be positive definite due to their non-convex nature. Moreover, having in mind that we work with noisy approximations \eqref{eq.sampled-fungrad}, imposing a strict decrease might be unnecessary. Thus, we employ a non-monotone trust-region (NTR) approach; see e.g. \cite{ahookhosh2012nonmonotone} or references therein.
Unlike the classical TR, our decision on acceptance of the trial point is not based only on the agreement between the model and the approximate objective function decrease, but on the independent subsampled function. This "control" function which is formed through \textit{additional sampling}, similar to one proposed in \cite{di2023lsos} for the line search framework, also has a role in controlling the sample average approximation error by adaptively choosing the sample size. Depending on the estimated progress of our algorithm, this can yield sample size scenarios ranging from mini-batch to full sample functions. We provide convergence analysis for all possible scenarios and show that the proposed method achieves almost sure convergence under standard assumptions for the TR framework such as Lipschitz-continuous gradients and bounded Hessian approximations. 

{\textbf{Literature review.}} Stochastic first-order methods such as stochastic gradient descent (SGD) method \cite{robbins1951stochastic,bottou2004large}, and its variance-reduced \cite{defazio2014saga,johnson2013accelerating,schmidt2017minimizing,nguyen2017sarah} and adaptive \cite{duchi2011adaptive, kingma2017adam} variants have been widely used in many ML and DL applications likely because of their proven efficiency in practice. However, due to the use of only gradient information, these methods come with several issues like, for instance, relatively slow convergence, high sensitivity to the hyper-parameters, stagnation at high training loss \cite{bottou2018optimization}, difficulty in escaping saddle points, and suffering from ill-conditioning \cite{kylasa2019gpu}. To cope with some of these issues, there are some attractive alternatives as well as their stochastic variants aimed at incorporating second-order information, e.g. Hessian-Free methods \cite{martens2010deep, Ma12, bollapragada2019exact, xu2020second, martens2015optimizing, goldfarb2020practical} which find an estimation of the Newton direction by (subsampled) Hessian-vector products without directly calculating the (subsampled) Hessian matrix, and limited memory Quasi-Newton methods which construct some approximations of the true Hessian only by using gradient information. Furthermore, algorithms based on Quasi-Newton methods have been the subject of many research efforts both in convex (see e.g. \cite{mokhtari2015global,gower2016stochastic} and references therein) and non-convex settings (see e.g. \cite{wang2017stochastic, berahas2020robust, bollapragada2018progressive} and references therein), or \cite{jahani2020scaling,berahas2022quasi} where the advantage of modern computational architectures and parallelization for evaluating the full objective function and its derivatives is employed. In almost all these articles, the Quasi-Newton Hessian approximation used is BFGS or limited memory BFGS (L-BFGS) with positive definiteness property, which is often considered in the line-search framework except e.g. \cite{rafati2018improving,yousefi2022stochastic}. A disadvantage of using BFGS may occur when it tries to approximate the true Hessian of a non-convex objective function in DL optimization. 
We refer to \cite{erway2020trust} as one of the earliest works in which the limited memory Quasi-Newton update SR1 (L-SR1) allowing for indefinite Hessian approximation was used in a trust-region framework with a periodical progressive overlap batching.

The potential usefulness of non-monotonicity may be traced back to \cite{grippo1986nonmonotone} where a non-monotone line-search technique was proposed for the Newton method to relax some standard line-search conditions and to avoid slow convergence of a deterministic method. Similarly, the idea of non-monotonicity exploited for trust-region could be dated back to \cite{deng1993nonmonotonic} and later e.g. \cite{ahookhosh2012nonmonotone, cui2011combining} for a general unconstrained minimization problem. This idea was also used for solving problems such as \eqref{eq.problem} in a stochastic setting; in \cite{krejic2015nonmonotone}, a class of algorithms was proposed that uses non-monotone line-search rules fitting a variable sample scheme at each iteration. In \cite{yousefi2022stochasticNTR}, a non-monotone trust-region algorithm using fixed-size subsampling batches was proposed for solving \eqref{eq.problem}. Recently, in \cite{sun2023trust} a noise-tolerant TR algorithm has been proposed at the time of writing this paper, in which both the numerator and the denominator of the TR reduction ratio are relaxed. The convergence analysis presented in \cite{sun2023trust} is based on the assumption that errors in function and gradient evaluations are bounded and do not diminish as the iterates approach the solution. In \cite{Cao2023BerahasScheinberg} the authors derive high probability complexity bounds for 
  first- and second-order trust-region methods with noisy oracles, where the targeted vicinity of the solution depends on the quality of the stochastic estimates of the objective function and its gradient. They analyze modified trust-region algorithms that utilize stochastic zeroth-order oracles both in the bounded noise case and the independent subexponential noise case.
Inspired by \cite{yousefi2022stochasticNTR}, in this work, we introduce a new second-order method in a subsampled non-monotone trust-region (NTR) approach, which works well with any Hessian approximation and employs an adaptive subsampling scheme. The foundation of our method differs from that of \cite{sun2023trust,Cao2023BerahasScheinberg}. Our method is based on an additional sampling strategy helping to control the non-martingale error due to the dependence between the non-monotone TR radius and the search direction. To the best of our knowledge, there are only a few approaches using additional sampling; see \cite{iusem2019variance,krejic2015descent,di2023lsos}. The rule that we apply in our method mainly corresponds to that presented in \cite{di2023lsos} where it is used in a line-search framework and plays a role in deciding whether to switch from the line-search to a predefined step size sequence or not. We adapted this strategy to the TR framework and used it to control the sample size in our method. It is worth pointing out that the additional sampling can be arbitrarily small, i.e. the sample size can even be $1$, and hence, it does not increase the computational cost significantly. Adaptive sample size strategies can be also found in other works, for instance, a type of adaptive subsampling strategy was applied for the STORM algorithm \cite{blanchet2019convergence,chen2018stochastic} which is also a second-order method in a standard TR framework. A different strategy using inexact restoration was proposed in \cite{bellavia2023stochastic} for a first-order standard TR approach. The variable size subsampling is not restricted to TR frameworks, see e.g. \cite{bollapragada2018progressive} where a progressive subsampling was considered for a line-search method.

{\textbf{Notation.}} Throughout this paper, vectors and matrices are respectively written in lowercase and uppercase letters unless otherwise specified in the context. The symbol $\triangleq$ is used to define a new variable. $\mathbb{N}$ and $\mathbb{R}^n$ denote the set of natural numbers and the real coordinate space of dimension $n$, respectively. The set of positive real numbers and non-negative integers are denoted by $\mathbb{R}^n_+$ and $\mathbb{N}_0$, respectively. Subscripts indicate the elements of a sequence. For a random variable $X$, the mathematical expectation of $X$ and the conditional expectation of $X$ given the $\sigma$-algebra $\mathcal{F}$ are respectively denoted by $\mathbb{E}(X)$ and $\mathbb{E}(X|\mathcal{F})$. The Euclidean vector norm  and the corresponding matrix norm are denoted  by $\|.\|$, while the cardinality of a set or the absolute value of a number is indicated by $|.|$.  Finally, \mbox {"a.s"} abbreviates the expression \mbox{"almost surely"}.

{\textbf{Outline of the paper.}} In Sect.~\ref{Algo}, we describe the algorithm and all the necessary ingredients. In Sect.~\ref{Conv}, we state the assumptions and provide an almost sure convergence analysis of the proposed method. \Cref{Num} is devoted to the numerical evaluation of a specific version of the proposed algorithm that makes use of an L-SR1 update and a simple sampling rule; we make a comparison with the state-of-the-art method STORM \cite{chen2018stochastic,blanchet2019convergence} to show the effectiveness of the proposed method for training DNNs in classification and regression tasks. A comparison with the popular first-order method ADAM \cite{kingma2017adam} is also presented. Some conclusions are drawn in Sect.~\ref{Concl}.

\section{The algorithm}\label{Algo}
 Within this section, we describe the proposed method called ASNTR (Adaptive Subsample  Non-monotone Trust-Region). At iteration $k$, given the current iteration $ w_k$, we form a quadratic model based on the subsampled function \eqref{eq.sampled-fungrad}, with $g_k \triangleq \nabla f_{\N_k}(w_k)$ and  an arbitrary  Hessian approximation $B_k$ satisfying 
\begin{equation} \label{eq.llBll} 
    \|B_k\| \leq L,
\end{equation}
for some $L>0$ and solve the common TR subproblem to obtain the relevant direction  
\begin{equation}\label{eq.TRsub}
 p_k \triangleq \arg\min_{p \in \mathbb{R}^n}  Q_k(p) \triangleq \frac{1}{2}p^T B_k p + g_k^T p \quad  \text{s.t.} \quad \left\| p \right\|_2 \leq \delta_k,
\end{equation}
for some TR radius $\delta_k>0$. We assume that at least some fraction of the Cauchy decrease is obtained, i.e., the direction satisfies 
    \begin{equation} \label{eq.Cauchy} 
    Q_k(p_k) \leq -\frac{c}{2} \|g_k\| \min \{\delta_k, \frac{\|g_k\|}{\|B_k\|}\},
    \end{equation}
for some $c \in (0,1]$. This is a standard assumption in TR and it is not restrictive even in the stochastic framework. In the classical deterministic TR approach, the trial point $ w_t = w_k + p_k $ acceptance is based on the agreement between the decrease in the function and that of its quadratic model. However, since we are dealing with noisy approximations \eqref{eq.sampled-fungrad}, we modify the acceptance strategy as follows. Motivated by the study in \cite{yousefi2022stochasticNTR}, we propose a non-monotone TR (NTR) framework instead of the standard TR one because we do not want to impose a strict decrease in the approximate function. Therefore, we define the relevant ratio as follows 
\begin{equation}\label{eq.rho_N}
    \rho_{\N_k}\triangleq\frac{f_{\N_k}(w_t)-r_{\N_k}}{Q_k(p_k)}, 
\end{equation}
where  
\begin{equation}\label{eq.r_N}
    r_{\N_k}\triangleq f_{\N_k}(w_k)+ t_k \delta_k,\quad t_k>0,
\end{equation}
and 
\begin{equation} \label{eq.sumT} 
\sum_{k=0}^{\infty} t_k \leq t < \infty.
\end{equation}
The sequence $ \{t_k\} $ is a summable sequence of positive numbers and one can define it in different ways to control the level of non-monotonicity in each iteration. Different choices of $\{t_k\}$ lead to different upper bounds of its sum which we denote by $t$. We allow $\N_k$ to be chosen arbitrarily in ASNTR. However, even if we impose  uniform sampling with replacement to $\N$  such that it yields an unbiased estimator $f_{\N_k}(w_k)$ of the objective function at $w_k$, the dependence between $\N_k$ and $w_t$ may produce  a biased estimator  $f_{\N_k}(w_t)$ of $f(w_t)$. This is because $\N_k$ directly affects the TR model $Q_k(p)$ through the approximate gradient $g_k= \nabla f_{\N_k}(w_k)$ and thus $p_k$ also depends on the choice of $\N_k$. Thus, $w_t=w_k+p_k$ is also dependent on $\N_k$. To overcome this difficulty, we apply an additional sampling strategy  \cite{di2023lsos}. To this end, at every iteration at which $N_k<N$, we choose another independent subsample set represented by the index set $\D_k \subset \N$ of size $D_k=|\D_k|<N$ and calculate $f_{\D_k}(w_k), f_{\D_k}(w_t)$ and $\bar{g}_k \triangleq \nabla f_{\D_k}(w_k)$ (see lines 5-6 of the ASNTR algorithm). There are no theoretical requirements on the size of $\D_k$, and hence, the additional sampling might be done cheaply, i.e. with a modest number of additional samples. In fact, in our experiments, we set $D_k=1$ for all $k$. Furthermore, in the spirit of TR, we define a linear model as $L_k(v)\triangleq v^T \bar{g}_k,$
and consider another agreement measure defined as follows
\begin{equation} \label{eq.rho_D}
    \rho_{\D_k} \triangleq \frac{f_{\D_k}(w_t)-r_{\D_k}}{L_k(-\bar{g}_k)},
\end{equation}
where 
\begin{equation} \label{eq.r_D} 
    r_{\D_k}\triangleq f_{\D_k}(w_k)+\delta_k \tilde{t}_k,\quad \tilde{t}_k>0,
\end{equation}
and
\begin{equation} \label{eq.sumTtilda}
    \sum_{k=0}^{\infty} \tilde{t}_k \leq \tilde{t} < \infty.
\end{equation}
We assume that $ \{\tilde{t}_k\} $ is a summable sequence of positive numbers and that $\tilde{t}$ is an upper bound of the sum of $ \{\tilde{t}_k\} $. Therefore, $\tilde{t}$   is controllable through the choice of $ \{\tilde{t}_k\} $. Notice that choosing a greater $\tilde{t}_k$ yields more chances for the trial point $w_t$ to be accepted. The denominator in \eqref{eq.rho_D} is the linear model computed along the negative gradient  $ \bar{g}_k $ and thus it is negative. Therefore, the condition $\rho_{\D_k}\geq \nu$ corresponds to  Armijo-like condition for the function $ f_{\D_k}, $ similar to one in \cite{di2023lsos}  i.e.,
$f_{\D_k}(w_{t}) \leq f_{\D_k}(w_{k})- \nu \| \nabla f_{\D_k}(w_{k})\|^2+\delta_k \tilde{t}_k$.
If $N_k<N$, the trial point is accepted only if both $\rho_{\N_k}$ and $\rho_{\D_k}$ are bounded away from zero; otherwise, if the full sample is used, the decision is made by $\rho_{\N_k}$ solely as in deterministic NTR (see lines 23-35 of the ASNTR algorithm). The rationale behind this is the following: we double-checked that the trial point obtained employing $f_{\N_k}$ is acceptable also with respect to another approximation of the objective function $f_{\D_k}$. Notice that $\D_k$ is chosen with replacement, uniformly and randomly from $\N$, independently from the choice of $\N_k$, and after the trial point $w_t$ is already determined. Therefore, conditionally on $\sigma$-algebra generated by all the previous choices of $\N_j, j=1,...,k$ and $\D_j, j=1,...,k-1$, the approximation  $f_{\D_k}(w_t)$ is an unbiased estimator of $f(w_t)$. 

Another role of $\rho_{\D_k}$ is to control the sample size. If the obtained trial point $w_t$ yields an uncontrolled increase in $f_{\D_k}$ in a sense that $\rho_{\D_k}< \nu$, i.e., $f_{\D_k}(w_{t}) > f_{\D_k}(w_{k})- \nu \| \nabla f_{\D_k}(w_{k})\|^2+\delta_k \tilde{t}_k$,  we conclude that we need a better approximation of the objective function and we increase the sample size $N_k$ by choosing a new sample set $\N_k$ for the next iteration.  Roughly speaking, an uncontrolled increase in $f_{\D_k}$ is possible if the approximate function $ f_{\D_k} $ is very different from $f_{\N_k}$ given that the search direction is computed for $f_{\N_k}$. The sample can also be increased if we are too close to a stationary point of the approximate function $f_{\N_k}$. This is stated in line 7 of ASNTR, where $h$ represents an SAA error estimate given by 
$$h(N_k)\triangleq \frac{N-N_k}{N}.$$

\begin{algorithm}[H]
\small
\begin{algorithmic}[1]
\State{Initialization: 
Choose $\mathcal{N}_0\subseteq \mathcal{N}$}. Set $k=0$, 
$\{t_k\} \in \mathbb{R}^{\infty}_{+}$ satisfying \eqref{eq.sumT}, $\{\tilde{t}_k\} \in \mathbb{R}^{\infty}_{+}$ satisfying \eqref{eq.sumTtilda}, 
$\delta_0$ and $\delta_{max} \in (0,\infty)$, 
$\epsilon \in [0, \frac{1}{2})$, 
$\nu \in (0,1/4)$,
$0<\tau_1 \leq 0.5< \tau_2 <1< \tau_3$, 
$0 < \eta < \eta_2 \leq 3/4$, and $\eta_1 \in (\eta, \eta_2)$.
\State{Given $f_{\N_k}(w_k)$, $g_k$ and $B_k$ satisfying \eqref{eq.llBll}, solve \eqref{eq.TRsub} for $p_k$ such that \eqref{eq.Cauchy} holds, and define the trial iterate $w_t = w_k + p_k$.}
\State{Given $f_{\N_k}(w_t)$, compute $\rho_{\N_k}$} using \eqref{eq.rho_N}.
\If{$N_k<N$}
\State{Choose $\D_{k}\subset \mathcal{N} $ randomly and uniformly, with replacement}.
\State{Given $f_{\D_k}(w_k)$, $\nabla f_{\D_k}(w_k)$, and $f_{\D_k}(w_t)$, compute $\rho_{\D_k}$ using \eqref{eq.rho_D}}.
\If{$\|g_k\|<\epsilon h(N_k)$  }
\State{Increase $N_k$ to $N_{k+1}$ and choose $\N_{k+1}$.} 
\Else
\If{$\rho_{\D_k}<\nu$}
\State{Increase $N_k$ to $N_{k+1}$ and choose $\N_{k+1}$.}
\Else
\If{$\rho_{\N_k} < \eta$}
\State{Set $N_{k+1} = N_k$ and $\N_{k+1} = \N_k$.}
\Else
\State{Set $N_{k+1} = N_k$ and choose $\N_{k+1}$.}
\EndIf
\EndIf
\EndIf
\Else
\State{$N_{k+1} = N$}
\EndIf
\If{$N_k<N$} 
\If{$\rho_{\N_k} \geq \eta$ and $\rho_{\D_k} \geq \nu$} 
\State{$w_{k+1}=w_t$.} 
\Else
\State{ $w_{k+1}=w_k$.}
\EndIf
\Else 
\If{$\rho_{\N_k} \geq \eta$} 
\State{$w_{k+1}=w_t$.} 
\Else
\State{ $w_{k+1}=w_k$.}
\EndIf
\EndIf
\If{$\rho_{\N_k}<\eta_1$,}
\State{ $\delta_{k+1}=\tau_1\delta_k$}
\ElsIf{$\rho_{\N_k}>\eta_2$ and $\|p_k\|\geq \tau_2 \delta_k$,}
\State{$\delta_{k+1}=min\{\tau_3\delta_k, \delta_{max}\}$,}
\Else
\State{$\delta_{k+1}=\delta_k$.}
\EndIf
\If{Some stopping conditions hold}
\State{Stop training}
\Else
\State{Set $k=k+1$ and go to step 2.}
\EndIf
\end{algorithmic}
\caption{\small ASNTR\label{alg_ASNTR}}
\end{algorithm}

The algorithm can also keep the same sample size (see lines 14 and 16 of the ASNTR algorithm). Keeping the same sample $\N_k$ in line 14 corresponds to the case where the trial point is acceptable with respect to $f_{\D_k}$, but we do not have a decrease in $f_{\N_k}$. In that case, the (non-monotone) TR radius ($\delta_k$) is decreased. Otherwise, we allow the algorithm in line 16 to choose a new sample of the current size and exploit some new data points. The strategy for updating the sample size is described in lines 7-19 of the ASNTR algorithm.

Notice that the sample size cannot be decreased, and if the full sample is reached it is kept until the end of the procedure. Moreover, it should be noted that ASNTR provides complete freedom in terms of the increment in the sample size as well as the choice of samples $\N_k$. This leaves enough space for different sampling strategies within the algorithm. As we already mentioned, mostly depending on the problem, ASNTR can result in a mini-batch strategy, but it can also yield an increasing sample size procedure. 

The TR radius is updated within lines 36-42 of ASNTR. We follow a common update strategy for TR approaches. It is completely based on $f_{\N_k}$ since it is related to the error of the quadratic model, and not to the SAA error which we control by additional sampling. Thus, if the $\rho_{\N_k} $ is small we decrease the trust-region size, lines 36-37. Otherwise, the trust-region is either enlarged or kept the same, lines 38-42 of ASNTR. Notice that the additional condition that relates the norm of the search direction $ p_k $ and current trust-region size $ \delta_k $ does not play any role in the theoretical analysis but it is important in practical implementation. We need some predetermined hyper-parameters for ASNTR, which are established in the algorithm's initial line according to ones outlined in relevant references e.g.,  \cite{nocedal2006numerical,erway2020trust}, and to meet the assumptions underlying the convergence analysis.
\section{Convergence analysis}\label{Conv}
We make the following standard assumption for the TR framework needed to prove the a.s. convergence result of ASNTR.

\begin{assumption}\label{ass1}  The functions $f_i, i=1,...,N$ are bounded from below and  twice continuously-differentiable   with  L-Lipschitz-continuous gradients. 
\end{assumption}
First, we prove an important lemma that will help us prove the main result, the a.s. convergence of ASNTR. There are two possible scenarios depending on the size of the sample sequence:  1) mini-batch scenario - where the subsampling is employed in each iteration; 2) deterministic scenario - where the full sample is reached eventually. We start the analysis by considering the first, mini-batch scenario.  In Lemma \ref{Lemma2new}, we show that in this case there holds  $\rho_{\D_k}\geq \nu$ for any realization of $\D_k$ and all $k$ sufficiently large. 

Let us denote by $\D_k^{+}$ the subset of all possible outcomes of $\D_k$ at iteration $k$ that satisfy $\rho_{\D_k} \geq \nu$, i.e., 
\begin{equation} \label{dkp}
    \D_k^{+}= \{\D_k \subset \N \; | \; f_{\D_k}(w_{t}) \leq f_{\D_k}(w_{k})- \nu \| \nabla f_{\D_k}(w_{k})\|^2+\delta_k \tilde{t}_k \}.
\end{equation}
Notice that if $\D_k \in \D_k^{+}$ and $\rho_{\N_k} \geq \eta$ then $w_{k+1}=w_t$. 
On the other hand, if $\D_k \in \D_k^{+}$
and $\rho_{\N_k} < \eta$ then $w_{k+1}=w_k$. Finally,  if $\D_k \in \D_k^{-}$, where 
\begin{equation} \label{dkm}
    \D_k^{-}= \{\D_k \subset \N\; | \; f_{\D_k}(w_{t}) > f_{\D_k}(w_{k})- \nu \| \nabla f_{\D_k}(w_{k})\|^2+\delta_k \tilde{t}_k \},
\end{equation}
we have again $w_{k+1}=w_k$.

\begin{lemma}\label{Lemma2new} 
    Suppose that \Cref{ass1} holds. If $N_k<N$ for all $k \in \mathbb{N}$, then there exists $k_1 \in \mathbb{N}$ such that $\rho_{\D_k}\geq \nu$ for all $k \geq k_1$ and for all possible realizations  $\D_k$.
\end{lemma}
\begin{proof} Assume that $N_k<N$ for all $k \in \mathbb{N}$. So, there exists some $\bar{N}<N$  and $k_0\in \mathbb{N}$ such that $N_k=\bar{N}$ for all $k \geq k_0$. Now, let us assume that$\rho_{\D_k}\geq \nu$ for at least one possible outcome of $ \D_k. $ 
This implies the existence of an infinite subsequence of iterations $K \subseteq \mathbb{N}$ such that $\rho_{\D_k}\geq \nu$ is not satisfied for at least one possible outcome of $\D_k$.
To be more precise,  let us use the notation as in \eqref{dkp}-\eqref{dkm}. Then, we have that  $\D_k^{-}\neq \emptyset$ for all $k \in K$. Since $\D_k$ is chosen randomly and uniformly with replacement from the finite set $\N$ and  $D_k\leq N-1$, we know that each $\D_k$ has only finitely many possible outcomes $S(D_k)\leq \bar{S}:=(2N-2)!/((N-1)!)^2$ and  there exists  $\underline{p} \in (0,1)$ such that  $P(\D_k \in \D_k^{-})\geq \underline{p}$, i.e.,  $P(\D_k \in \D_k^{+})\leq 1-\underline{p}=:\bar{p}<1$ for all $k \in K$. So, 
$$P(\D_k \in \D_k^{+}, k \in K)\leq \prod_{k \in K} \bar{p}=0. $$
Therefore, a.s. there exists an iteration $ k \geq k_1 $ such that  $\rho_{\D_k} < \nu $ and the sample size is increased, 
which is in contradiction with $N_k=\bar{N}$ for all $k$ large enough. This completes the proof.
\end{proof}

Next, we prove that the mini-batch scenario implies a non-monotone type of decrease for all $k$ large enough. 

\begin{lemma}\label{Lemmafdescent1} 
    Suppose that \Cref{ass1} is satisfied  and  
    $N_k<N$ for all $k \in \mathbb{N}$. Then
    $$f(w_t) \leq f(w_k)- \nu \| \nabla f(w_{k})\|^2+\delta_k \tilde{t}_k,$$
    holds for all $k \geq k_1$, where $k_1$ is as in Lemma \ref{Lemma2new}.
\end{lemma}
\begin{proof} Lemma \ref{Lemma2new} implies  that $\rho_{\D_k}\geq \nu$, i.e., 
\begin{equation} \label{rev1} f_{\D_k}(w_{t}) \leq f_{\D_k}(w_{k})- \nu \| \nabla f_{\D_k}(w_{k})\|^2+\delta_k \tilde{t}_k,\end{equation} 
for all $k \geq k_1$ and for all possible realizations  of $\D_k$. Since the sampling of $\D_k$ is with replacement, notice that this further yields
\begin{equation} \label{rev2} f_{i}(w_{t}) \leq f_{i}(w_{k})- \nu \| \nabla f_{i}(w_{k})\|^2+\delta_k \tilde{t}_k,
\end{equation}
for all $i \in \N$ and all $k \geq k_1$. Indeed, if there exists $i \in \N$ such that \eqref{rev2} is violated, then \eqref{rev1} is violated for at least one possible  realization of $\D_k$, namely, for  $\D_k=\{i,i,...,i\}$.  Thus, for all $k \geq k_1$ there holds 
\begin{eqnarray} 
f(w_t) &=& \frac{1}{N} \sum_{i=1}^{N} f_{i}(w_{t}) \leq \frac{1}{N} \sum_{i=1}^{N} (f_{i}(w_{k})- \nu   \| \nabla f_{i}(w_{k})\|^2+\delta_k \tilde{t}_k)\\\nonumber
&=& f(w_k)- \nu \frac{1}{N} \sum_{i=1}^{N} \| \nabla f_i(w_{k})\|^2+\delta_k \tilde{t}_k\\\nonumber
&\leq & f(w_k)- \nu \| \nabla f(w_k)\|^2+\delta_k \tilde{t}_k,
\end{eqnarray}
where the last inequality comes from the fact that $\|\cdot\|^2$ is convex and therefore
$$\|\nabla f(w_k)\|^2=\|\frac{1}{N} \sum_{i=1}^{N} \nabla f_i(w_{k})\|^2\leq \frac{1}{N} \sum_{i=1}^{N} \|\nabla f_i(w_{k})\|^2.$$
\end{proof}

Now, let us prove that starting from some finite but random iteration,  the sequence of iterates generated by ASNTR belongs to a level set. This level set is also random since it depends on the sample path.  
\begin{lemma}\label{Lemma1} 
    Suppose that \Cref{ass1} holds. Then  
    $$f(w_{\tilde{k}+k})\leq f(w_{\tilde{k}}) + \delta_{max}\max \{t,\tilde{t}\}, \quad k = 0,1,\ldots\, $$
    where $\tilde{k}$ is some finite random number,  $t$ and $\tilde{t}$ correspond to those in \eqref{eq.sumT} and \eqref{eq.sumTtilda} respectively.
\end{lemma}
\begin{proof}  Let us consider both scenarios, mini-batch and deterministic separately. In the mini-batch case,  when ${N_k<N}$  for each ${k}$, Lemma \ref{Lemmafdescent1} implies that 
    $$f(w_t) \leq f(w_k)- \nu \| \nabla f(w_{k})\|^2+\delta_k \tilde{t}_k\leq f(w_k)+\delta_k \tilde{t}_k,$$
holds for all $k \geq k_1$, where $k_1$ is as in Lemma \ref{Lemma2new}. Since $w_{k+1} = w_t $ or  $w_{k+1}=w_k$, there holds 
    $$f(w_{k+1}) \leq f(w_k)+\delta_k \tilde{t}_k,$$ 
for all $k \geq k_1$. So, the summability of $\{\tilde{t}_k\}$ \eqref{eq.sumTtilda} and the fact that $\delta_k \leq \delta$ for all $k$ together imply that the statement of this lemma holds with $\tilde{k}=k_1$.

Assume now that  ${N}$ is achieved at some finite iteration, i.e.,  there exists a finite iteration $k_2$ such that $N_k=N$ for all $k \geq  k_2$. Thus, the trial point for all $k \geq  k_2$ is accepted if and only if $\rho_{\N_k}\geq \eta$. This implies that  for all $k \geq  k_2$ we either have $f(w_{k+1})= f(w_k)$ or
\begin{equation} \label{eq.inLemma4} 
f(w_{k+1}) \leq f(w_k)+\delta_k t_k-\frac{\eta c}{2}\|g_k\| \min\{\delta_k, \frac{\|g_k\|}{\|B_k\|}\}\leq f(w_k)+\delta_{max}t_k,
\end{equation}
where $\|g_k\| = \|\nabla f(w_k)\|$ in this scenario. Thus, using the summability of $\{t_k\}$ in \eqref{eq.sumT} we obtain the result with $\tilde{k}=k_2$. 
\end{proof} 

In order to prove the main convergence result, we assume that the expected value of $f(w_{\tilde{k}})$ is uniformly bounded, where the expectation is taken for all possible sample paths. This assumption, together with the result of Lemma \ref{Lemma1}, implies that the sequence $\{f(w_k)\}_{k \geq \tilde{k}}$ is uniformly bounded in expectation. 
Indeed, by assuming $\mathbb{E}(|f(w_{\tilde{k}})| ) \leq C$ we have $\mathbb{E}(|f(w_{\tilde{k}})|\; | \; A) \leq C_1$, where $A$ represents the subset of all possible outcomes (sample paths) such that the full sample is reached eventually and $C_1$ is some positive constant. 
$$P(A) \mathbb{E}(|f(w_{\tilde{k}})|\; | \; A)\leq P(A) \mathbb{E}(|f(w_{\tilde{k}})|\; | \; A)+P(\bar{A}) \mathbb{E}(|f(w_{\tilde{k}})|\; | \; \bar{A})= \mathbb{E}(|f(w_{\tilde{k}})| ) \leq C,$$
where $\bar{A}$ represents all possible sample paths that remain in the mini-batch scenario. Thus, we obtain 
$$\mathbb{E}_A (|f(w_{\tilde{k}})| ):=\mathbb{E}(|f(w_{\tilde{k}})|\; | \; A)\leq C/P(A)=:C_1.$$
Similarly, 
$$\mathbb{E}_{\bar{A}} (|f(w_{\tilde{k}})| ):=\mathbb{E}(|f(w_{\tilde{k}})|\; | \; \bar{A})\leq C/P(\bar{A})=:C_2.$$
Notice that the above-mentioned assumption is satisfied under the assumption of bounded iterates which is fairly common in a stochastic optimization framework. 
\begin{theorem} \label{Theorem1} 
Suppose that \Cref{ass1} holds and that $\mathbb{E}(|f(w_{\tilde{k}})|) \leq C< \infty$, where $\tilde{k}$ is as in Lemma \ref{Lemma1}. Then the sequence  $\{w_k\}$ generated by ASNTR algorithm satisfies  
$$\liminf_{k \to \infty} \|\nabla f(w_k)\|=0 \quad \mbox{a.s}.$$
\end{theorem}
\begin{proof} Assume two different scenarios with $N_k=N $ for all $k$ large enough, and $N_k<N$ for all $k$. Let us start with the first one in which $N_k=N$ for all $k \geq\tilde{k}$, where  $\tilde{k}$ is random but finite. In this case, ASNTR reduces to the non-monotone \textit{deterministic} trust-region algorithm applied on $f$. By $L$-Lipschitz continuity of $\nabla f$, we obtain 
\begin{equation}
        \label{5n1} 
        |f(w_k+p_k)- f(w_k)-\nabla^T f(w_k) p_k| \leq \frac{L}{2} \|p_k\|^2.
\end{equation}
Now, let us denote $\rho_k \triangleq \rho_{\N_k}$ and assume that $\|\nabla f(w_k)\| \geq \varepsilon >0$ for all $k\geq \tilde{k}$. Then, 
\begin{eqnarray} \label{5n2} 
    |\rho_k-1| &=& 
    \frac{|f(w_k+p_k)- f(w_k)-\delta_k t_k-\nabla^T f(w_k) p_k-0.5 p_k^T B_k p_k|}{|Q_k(p_k)|}\\\nonumber
    &\leq & \frac{0.5 L\|p_k\|^2 + \delta_k t_k+0.5 L \|p_k\|^2}{0.5 c  \|g_k\| \min \{\delta_k, \frac{\|g_k\|}{\|B_k\|}\}}\\\nonumber
    &\leq & \frac{L\|p_k\|^2 + \delta_k t_k}{0.5 c  \varepsilon \min \{\delta_k, \frac{\varepsilon}{L}\}},
\end{eqnarray}
where $\|g_k\| = \|\nabla f(w_k)\|$ in this scenario. Let us define $\tilde{\delta}=\frac{\varepsilon c }{20 L}$. Without loss of generality, given that the sequence $\{t_k\}$ is summable and hence $ t_k \to 0$ (see \eqref{eq.sumT}), we can assume that $t_k \leq L \tilde{\delta}$ for all $k\geq\tilde{k}$. Observe now the  iterations $ k $ for  $k > \tilde{k}$. If $\delta_k \leq \tilde{\delta}$, then $\delta_{k+1} \geq \delta_k$. This is due to the fact  $\tilde{\delta} \leq \frac{\varepsilon}{L}$,
\begin{eqnarray} \label{5n3} 
    |\rho_k-1| \leq 
     \frac{L\delta_k^2 + \delta_k t_k}{0.5 c  \varepsilon \delta_k}\leq  
     \frac{2 L\tilde{\delta} }{0.5 c  \varepsilon } < \frac{1}{4},
\end{eqnarray}
i.e., $\rho_k>\frac{3}{4}$ which implies that the NTR radius is not decreased; see lines 36-42 in ASNTR. Further, there exists $\hat{\delta}>0$ such that $\delta_k \geq \hat{\delta}$ for all $k \geq \tilde{k}$. Moreover, for all $k \geq\tilde{k}$, the assumption $\rho_k<\eta$, where $\eta<\eta_1$, would yield a contradiction since it would imply $\lim_{k \to \infty} \delta_k=0$. Therefore, there must exist an infinite set $K\subseteq \mathbb{N}$ as $K=\{ k \geq \tilde{k} \,| \,\rho_k \geq \eta \}$. For all $k \in K$, we have $$f(w_{k+1}) \leq f(w_k)+\delta_k t_k-\frac{c}{8}\|g_k\| \min\{\delta_k, \frac{\|g_k\|}{\|B_k\|}\}\leq f(w_k)+\delta_{max}t_k-\frac{c}{8} \varepsilon \min\{\hat{\delta}, \frac{\varepsilon}{L}\}.$$ Now, let $\bar{c} \triangleq\frac{c}{8} \varepsilon \min\{\hat{\delta}, \frac{\varepsilon}{L}\}$. Since $t_k$ tends to zero, we have $\delta_{max}t_k \leq \frac{\bar{c}}{2}$ for all $k$ large enough. Without loss of generality, we can say that this holds for all $k \in K$, rewriting $K=\{k_j\}_{j \in \mathbb{N}_0}$, we have $$f(w_{k_j+1})\leq f(w_{k_j})-\frac{\bar{c}}{2}.$$
Since $w_{k+1}=w_k$ for all $k\geq \tilde{k}$ and $k \notin K$, i.e. when $\rho_k < \eta$, we obtain 
$$f(w_{k_{j+1}})\leq f(w_{k_j})-\frac{\bar{c}}{2}.$$ 
Thus we obtain for all $j \in \mathbb{N}_0$ 
\begin{equation} \label{eq.inThe1rev1} 
     f(w_{k_{j}})\leq f(w_{k_0})- j \frac{\bar{c}}{2} \leq f(w_{\tilde{k}})+\delta_{max}\max \{t, \tilde{t}\}- j \frac{\bar{c}}{2},
\end{equation}
where the last inequality is due to \Cref{Lemma1}. Furthermore, applying the expectation and using the assumption $\mathbb{E}(f(w_{\tilde{k}})) \leq C$ we get 
 \begin{equation} \label{eq.inThe1} 
    \mathbb{E}_{A}(f(w_{k_{j}}))\leq C_1+\delta_{max}\max \{t, \tilde{t}\}-j \frac{\bar{c}}{2}.
\end{equation}
Letting $j$ tend to infinity in \eqref{eq.inThe1}, we obtain $\lim_{j \to \infty} \mathbb{E}_{A}(f(w_{k_j}))=-\infty$, which is in contradiction with the assumption of $f$ being bounded from below. Therefore,  $\|\nabla f(w_k)\| \geq \varepsilon >0$ for all $k\geq \tilde{k}$ can not be true, so we conclude that $$\liminf_{k \to \infty} \|\nabla f(w_k)\|=0.$$

Now let us consider the mini-batch scenario, i.e., $N_k < N$ for all $k$, i.e., the sample size is increased only finitely many times. According to Lemma \ref{Lemma2new} and lines 7-8 of the algorithm,  
the currently considered scenario implies the existence of a finite $\tilde{k}_1$ such that 
\begin{equation} \label{5new1} N_k=\tilde{N}, \quad \|g_k\| \geq \epsilon h(N_k)\triangleq\epsilon_{\tilde{N}}>0 \quad \mbox{and} \quad \rho_{\D_k}\geq \nu,\end{equation}
for all $k\geq \tilde{k}_1$ and some $\tilde{N}<N.$ Now, let us prove that there exists an infinite subset of iterations $\tilde{K} \subseteq \mathbb{N}$ such that  $\rho_{k}\geq \eta$ for all $k \in \tilde{K}$, i.e., the trial point $w_t$ is accepted infinitely many times. Assume a contrary, i.e., there exists some finite $\tilde{k}_2$ such that $\rho_{k}<\eta$ for all $k \geq \tilde{k}_2$. Since $\eta < \eta_1$, this further implies that $\lim_{k \to \infty}\delta_k=0$; see lines 36 and 37 in \Cref{alg_ASNTR}. Moreover, line 13 of \Cref{alg_ASNTR} implies that the sample does not change, meaning that there exists $\tilde{\N}\subset \N$ such that $\N_k=\tilde{\N}$ for all $k \geq \tilde{k}_3\triangleq\max\{\tilde{k}_1, \tilde{k}_2\}$. By $L$-Lipschitz continuity of $\nabla f_{\tilde{\N}}$, we obtain 
    \begin{equation}
        \label{5new2} 
        |f_{\tilde{\N}}(w_k+p_k)- f_{\tilde{\N}}(w_k)-\nabla^T f_{\tilde{\N}}(w_k) p_k| \leq \frac{L}{2} \|p_k\|^2.
    \end{equation}
    For every $k \geq \tilde{k}_3$, then we have 
    \begin{eqnarray} \label{5new3} 
    |\rho_{k}-1| &=& 
    \frac{|f_{\tilde{\N}}(w_k+p_k)- f_{\tilde{\N}}(w_k)-\delta_k t_k-\nabla^T f_{\tilde{\N}}(w_k) p_k-0.5 p_k^T B_k p_k|}{|Q_k(p_k)|}\\\nonumber
    &\leq & \frac{0.5 L\|p_k\|^2 + \delta_k t_k+0.5 L \|p_k\|^2}{0.5 c \|g_k\| \min \{\delta_k, \frac{\|g_k\|}{\|B_k\|}\}}\\\nonumber
    &\leq & \frac{L\delta_k^2 + \delta_k t_k}{0.5 c \epsilon_{\tilde{N}} \min \{\delta_k, \frac{\varepsilon}{L}\}}.
    \end{eqnarray}
Since $\lim_{k \to \infty}\delta_k=0$, there exists $\tilde{k}_4$ such that for all $k \geq \tilde{k}_4$ we obtain 
$$|\rho_{k}-1|\leq \frac{L\delta_k^2 + \delta_k t_k}{0.5 c \epsilon_{\tilde{N}}  \delta_k}=\frac{L\delta_k + t_k}{0.5 c  \epsilon_{\tilde{N}} }.$$
Due to the fact that $t_k$ tends to zero, we obtain that $\lim_{k \to \infty} \rho_{k}=1$ which is in contradiction with $\rho_{k} <\eta <1/4.$ Thus, we conclude that there must exist $\tilde{K} \subseteq \mathbb{N}$ such that $\rho_{k}\geq \eta$ for all $k \in \tilde{K}$. Let us define $K_1 \triangleq \tilde{K}\cap \{\tilde{k}_1,\tilde{k}_1+1,\ldots\}$. Notice that we have $\rho_{\D_k}\geq \nu$ and $\rho_{k}\geq \eta$ for all $k \in K_1$. Thus, due to Lemma \ref{Lemmafdescent1},  the following holds for all $k \in K_1$
   \begin{equation} \label{p9}
   \begin{split}
        f(w_{k+1})=f(w_t) 
        &\leq f(w_k)-\nu\|\nabla f(w_{k})\|^2 + \delta_k \tilde{t}_k, \\
        &\leq f(w_k)-\nu\|\nabla f(w_{k})\|^2 + \delta_{max} \tilde{t}_k.
   \end{split}
    \end{equation}
    
Notice that $w_{k+1}=w_k$ in all other iterations $k \geq \tilde{k}_1$ and $k \notin K_1$. Rewriting $K_1=\{k_j\}_{j \in \mathbb{N}_0}$, for all $j \in \mathbb{N}_0$, we obtain
    $$ f(w_{k_{j+1}})=f(w_{k_{j}+1})
    \leq f(w_{k_j})-\nu  \|\nabla f(w_{k_j})\|^2+ \delta_{max} \tilde{t}_{k_j}.$$
    Then, due to the summability of the sequeence $\{\tilde{t}_k\}$ in \eqref{eq.sumTtilda}, and Lemma \ref{Lemma1}, the following holds for any $s \in \mathbb{N}_0$
    \begin{eqnarray*}
        \begin{split}     
            f(w_{k_{s+1}})
            &\leq f(w_{k_0})-\nu \sum_{j=0}^{s}\|\nabla f(w_{k_j})\|^2+ \delta_{max} \tilde{t} \\
            &\leq f(w_{\tilde{k}})+\delta_{max}\max \{t,\tilde{t}\}-\nu  \sum_{j=0}^{s}\|\nabla f(w_{k_j})\|^2+ \delta_{max} \tilde{t}.
        \end{split}
    \end{eqnarray*}

    Now, applying the expectation, using $\mathbb{E} (f(w_{\tilde{k}}))\leq C$ which implies $\mathbb{E}_{\bar{A}} (|f(w_{\tilde{k}})| )\leq C_2$, and the boundedness of $f$ from below,  letting $s$ tend to infinity in the above inequality yields 
    $$ \sum_{j=0}^{\infty} \mathbb{E}_{\bar{A}}(\|\nabla f(w_{k_j})\|^2) <\infty.$$
    Moreover, the following holds as a consequence of an extended version of Markov's inequality with the  specific choice of the nonnegative function $\Phi(y)=y^2$ nondecreasing on $ y\geq 0$: for any $\epsilon>0$
    $$ P_{\bar{A}}(\|\nabla f(w_{k_j})\| \geq \epsilon)\leq \frac{\mathbb{E}_{\bar{A}}(\Phi(\|\nabla f(w_{k_j})\|))}{\Phi(\epsilon)} = \frac{\mathbb{E}_{\bar{A}}(\|\nabla f(w_{k_j})\|^2)}{\epsilon^2}.$$
    Therefore, we have
    $$ \sum_{j=0}^{\infty} P_{\bar{A}}(\|\nabla f(w_{k_j})\| \geq \epsilon)< \infty.$$
    Finally, Borel-Cantelli Lemma implies that (in the current scenario) almost surely $ \lim_{j \to \infty} \|\nabla f(w_{k_j})\|=0$. 
    Combining all together, the result follows as in both scenarios we have at least $$\liminf_{k \to \infty} \|\nabla f(w_{k})\|=0 \quad \mbox{a.s}.$$
    
\end{proof}
 
Finally, we analyze the convergence rate of the proposed algorithm for a restricted class of problems stated in the following theorem. 
\begin{theorem} \label{Theorem1convrate} 
Suppose that the assumptions of Theorem \ref{Theorem1} hold and that $\nu <L/2$. Moreover, suppose that $f $ is $m$-strongly convex and the sequence $\{\tilde{t}_k\}$ converges to zero R-linearly. If $N_k < N$ for all $k \in \mathbb{N}$, then the sequence  $\{w_k\}_{k \in K_1}$ with  $K_1$  as in \eqref{p9} converges  to the unique minimizer $w^*$ of  $f$ R-linearly in the  mean squared sense.   
\end{theorem}
\begin{proof}
    Considering the mini-batch scenario in the proof of the previous theorem, we obtain 
    $$ f(w_{k_{j+1}})
    \leq f(w_{k_j})-\nu  \|\nabla f(w_{k_j})\|^2+ \delta_{max} \tilde{t}_{k_j},$$
    for all $j \in \mathbb{N}$. Moreover, the strong convexity of $f$ implies 
    $$ f(w_{k_{j+1}})-f(w^*)
    \leq f(w_{k_j})-f(w^*)-\nu \frac{2}{L} (f(w_{k_j})-f(w^*))+ \delta_{max} \tilde{t}_{k_j},$$
    for all $j \in \mathbb{N}$.
    Now,  applying the expectation and defining
    $e_j:=\mathbb{E}_{\bar{A}}(f(w_{k_j})-f(w^*)),$ 
    from the previous inequality we obtain 
    $$ e_{j+1}
    \leq \theta e_j+ \varepsilon_j,$$
    where $ \varepsilon_j:=\delta_{max} \tilde{t}_{k_j}$ and $\theta:=1-2 \nu  /L \in (0,1)$ according to the assumption on $\nu$. Since the previous inequality holds for each $j$, we conclude that for all $j \in \mathbb{N}$ there holds
    $$e_j\leq \theta^j e_0+s_j,$$
    where $s_j:=\delta_{max}\sum_{i=1}^{j} \theta^{i-1} \varepsilon_{j-i}.$
    Notice that $\{\varepsilon_j\}$ converges to zero R-linearly according to the assumption on $\tilde{t}_k$. This further implies that $s_j$ converges to zero R-linearly (see Lemma 4.2 in \cite{krejic2015nonmonotone}). Thus, we conclude that  $\{e_j\}$ converges to zero R-linearly. 
    Finally, since the strong convexity also implies 
    $$\frac{m}{2} \mathbb{E}_{\bar{A}}(\|w_{k_j}-w^*\|^2)\leq e_j,$$ we obtain the result. 
\end{proof}

\section{Numerical experiments}\label{Num}               
In this section, we provide some experimental results to make a comparison between ASNTR and STORM (as Algorithm 5 in \cite{chen2018stochastic}). We examine the performance of both algorithms for training DNNs in two types of problems: (i) Regression with synthetic \texttt{DIGITS} dataset and (ii) Classification with \texttt{MNIST} and \texttt{CIFAR10} datasets.
\footnote{
Available online: \url{https://www.kaggle.com/datasets/hojjatk/mnist-dataset } (\texttt{MNIST}) \url{https://www.mathworks.com/help/deeplearning/ug/data-sets-for-deep-learning.html} (\texttt{DIGITS} and \texttt{CIFAR10})}

We also provide additional results to give more insights into the behavior of the ASNTR algorithm, especially concerning the sampling strategy. All experiments were conducted with the MATLAB DL toolbox on an Ubuntu 20.04.4 LTS (64-bit) Linux server VMware with 20GB memory using a VGPU NVIDIA A100D-20C. 
\subsection{Experimental configuration}\label{body3a}
All three image datasets are divided into training and testing sets including $N$ and $\hat{N}$ samples, respectively, and whose pixels in the range $[0, 255]$ are divided by 255 so that the pixel intensity range is bounded in $[0, 1]$ (zero-one rescaling). To define the initialized networks and training loops of both algorithms, we have applied \texttt{dlarray} and \texttt{dlnetwork} MATLAB objects.\footnote{A MATLAB-based tutorial on implementing custom loops for training a deep neural network is available here: \url{http://doi.org/10.13140/RG.2.2.33008.94720/2}} The networks' parameters in $w\in\mathbb{R}^n$ are initialized by the Glorot (Xavier) initializer \cite{glorot2010understanding} and zeros for respectively weights and biases of convolutional layers as well as ones and zeros respectively for scale and offset variables of batch normalization layers. \Cref{tab1} describes the hyper-parameters applied in both algorithms.

Since ASNTR and STORM allow the use of any Hessian approximations, the underlying quadratic model can be constructed by exploiting a Quasi-Newton update for $B_k$. Quasi-Newton updates aim at constructing Hessian approximations using only gradient information and thus incorporating second-order information without computing and storing the true Hessian matrix. We have considered the \textbf{L}imited \textbf{M}emory \textbf{S}ymmetric \textbf{R}ank \textbf{one} (L-SR1) update formula to generate $B_k$ rather than other widely used alternatives such as limited memory Broyden-Fletcher-Goldfarb-Shanno (L-BFGS). The L-SR1 updates might better navigate the pathological saddle points present in the non-convex optimization found in DL applications. Given $B_0 = \gamma_ k I$ at iteration $k$, and curvature pair $(s_k, y_k)=(w_t-w_k, g_t - g_k)$ where $g_k \triangleq \nabla f_{\N_k}(w_k)$ and $g_t \triangleq \nabla f_{\N_k}(w_t)$ provided that $(y_k - B_ks_k)^Ts_k\neq0$, the SR1 updates are obtained as follows
\begin{equation}\label{eq.sr1}
B_{k+1} = B_k + \dfrac{(y_k - B_k s_k)(y_k - B_k s_k)^T}{(y_k - B_k s_k)^T s_k}, \quad k=0,1,\ldots.
\end{equation}
Using two limited memory storage matrices $S_k$ and $Y_k$ with at most $l \ll n$ columns for storing the recent $l$ pairs $\left\{s_j, y_j\right\}$ \; $ j \in \{1,\ldots,l\}$, the compact form of L-SR1 updates can be written as 
\begin{equation}\label{eq.CompactSR1}
B_k = B_0 + \Psi_k  M_k \Psi_k^T, \quad k = 1,2,,\ldots,
\end{equation}
where 
$$\Psi_k   = Y_k - B_0S_k,\qquad M_k  = (D_k + L_k + L_k^T - S_k^T B_0 S_k)^{-1}$$ with matrices $L_k$ and $D_k$ which respectively are the strictly lower triangular part and the diagonal part of $S_k^TY_k$; see \cite{nocedal2006numerical}. Regarding the selection of the variable multiplier $\gamma_k$ in $B_0$, we refer to the initialization strategy proposed in \cite{erway2020trust}. Given the quadratic model $Q_k(p)$ using L-SR1 in ASNTR and STORM, we have used an efficient algorithm called OBS \cite{brust2017solving}, exploiting the structure of the L-SR1 matrix \eqref{eq.CompactSR1} and the Sherman-Morrison-Woodbury formula for inversions to obtain $p_k$.  

We have randomly (without replacement) chosen the index subset $\N_k \subseteq \{1,2,\dots, N\}$ to generate a mini-batch of size $N_k$ for computing the required quantities, i.e., subsampled functions and gradients. Given $N_0 = d+1$ where $d$ is the dimension of a single input sample $x_i \in \mathbb{R}^d$, we have adopted the linearly increased sampling rule that $N_k = \min( N, \max(100k + N_0, \lceil\frac{1}{{\delta_k}^2}\rceil)$ for STORM as in \cite{chen2018stochastic} while we have exploited the simple following sampling rule for ASNTR when the increase is necessary 
 \begin{equation}\label{eq.adSampling}
     N_{k+1}= \lceil 1.01 N_k\rceil,
 \end{equation}
 otherwise $N_{k+1}=N_k$. Using the non-monotone TR framework in our algorithm, we set $t_k = \frac{C_1}{(k+1)^{1.1}}$ and $\tilde{t}_k = \frac{C_2}{(k+1)^{1.1}}$ for some $C_1, C_2 > 0$ in \eqref{eq.r_N} and \eqref{eq.r_D}, respectively. We have also selected $\D_k$ with cardinality 1 at every single iteration in ASNTR.

In our implementations, each algorithm was run with 5 different initial random seeds. The criteria of both algorithms' performance (accuracy and loss) are compared against the number of gradient calls ($N_g$) during the training phase. The algorithms were terminated when $N_g$ reached the determined budget of the gradient evaluations ($N_g^{\text{max}}$). Each network is trained by ASNTR and STORM; then the trained network is used for the prediction of the testing dataset. Notice that the training loss and accuracy are the subsampled function's value and the number of correct predictions in percentage with respect to the given mini-batch.

\begin{table}[h]
\caption{\small Hyper-parameters}
\small
\centering
\begin{tabular}{||l|l||}\hline\hline
\multicolumn{1}{||c||}{\textbf{STORM}}\\\hline
{$\delta_0=1$,\, $\delta_{max}=10$, \,$l = 30$, \,$\eta_1=10^{-4}$,\, $\eta_2=10^{-3}$,\, $\gamma=2$}\\\hline
\multicolumn{1}{||c||}{\textbf{ASNTR}}\\\hline
{$\delta_0=1$,\, $\delta_{max}=10$, \,$l = 30$, \,$\eta=\nu=10^{-4}$,\, $\eta_1=0.1$,\, $\eta_2=0.75$,\, $\tau_1=0.5$,\, $\tau_2=0.8$,\,$\tau_3=2$}\\\hline\hline
\end{tabular}
\label{tab1}
\end{table}

\begin{table}[h]
\caption{\small Architectures of the networks.}
\footnotesize
\centering
\begin{tabular}{|l|l|}
\hline
{Regression} & \\
\hline
& $(Conv(3\times3@8,1,\text{same}){/}BN{/}ReLU{/}AvgPool(2\times2,2,0))$ \\
& $(Conv(3\times3@16,1,\text{same}){/}BN{/}ReLU{/}AvgPool(2\times2,2,0))$ \\
\textbf{\texttt{CNN-Rn}} & $(Conv(3\times3@32,1,\text{same}){/}BN{/}ReLU)$ \\
& $(Conv(3\times3@32,1,\text{same}){/}BN{/}ReLU{/}DropOut(0.2))$ \\
& $FC(1)$ \\
\hline\hline
{Classification} & \\
\hline
& $(Conv(3\times3@16,1,1){/}BN{/}ReLU)$ \\
& $B_1\begin{cases}
(Conv(3\times3@16,1,1){/}BN{/}ReLU)\\
(Conv(3\times3@16,1,1){/}BN) + \text{addition}(1){/}ReLU
\end{cases}$ \\
& $B_2\begin{cases}
(Conv(3\times3@16,1,1){/}BN{/}ReLU)\\
(Conv(3\times3@16,1,1){/}BN) + \text{addition}(1){/}ReLU
\end{cases}$ \\
& $B_3\begin{cases}
(Conv(3\times3@16,1,1){/}BN{/}ReLU)\\
(Conv(3\times3@16,1,1){/}BN) + \text{addition}(1){/}ReLU
\end{cases}$ \\
& $B_1\begin{cases}
(Conv(3\times3@32,2,1){/}BN{/}ReLU)\\
(Conv(3\times3@32,1,1){/}BN)\\
(Conv(1\times1@32,2,0){/}BN) + \text{addition}(2){/}ReLU
\end{cases}$ \\
\textbf{\texttt{ResNet-20}} & $B_2\begin{cases}
(Conv(3\times3@32,1,1){/}BN{/}ReLU)\\
(Conv(3\times3@32,1,1){/}BN) + \text{addition}(1){/}ReLU 
\end{cases}$ \\
& $B_3\begin{cases}
(Conv(3\times3@32,1,1){/}BN{/}ReLU)\\
(Conv(3\times3@32,1,1){/}BN) + \text{addition}(1){/}ReLU
\end{cases}$ \\
& $B_1\begin{cases}
(Conv(3\times3@64,2,1){/}BN{/}ReLU)\\
(Conv(3\times3@64,1,1){/}BN)\\
(Conv(1\times1@64,2,0){/}BN) + \text{addition}(2){/}ReLU\\
\end{cases}$ \\
& $B_2\begin{cases}
(Conv(3\times3@64,1,1){/}BN{/}ReLU)\\
(Conv(3\times3@64,1,1){/}BN) + \text{addition}(1){/}ReLU
\end{cases}$ \\
& $B_3\begin{cases}
(Conv(3\times3@64,1,1){/}BN{/}ReLU)\\
(Conv(3\times3@64,1,1){/}BN) + \text{addition}(1){/}gAvgPool{/}ReLU)
\end{cases}$ \\
& $FC(C{/}Softmax)$ \\
\hline\hline
{Classification} & \\
\hline
& $(Conv(5\times5@20,1,0){/}ReLU{/}MaxPool(2\times2,2,0))$ \\
\textbf{\texttt{LeNet-like}} & $(Conv(5\times5@50,1,0){/}ReLU{/}MaxPool(2\times2,2,0))$ \\
& $FC(500{/}ReLU)$ \\
& $FC(C{/}Softmax)$ \\
\hline
\multicolumn{2}{l}{
\begin{minipage}{13cm}
\vspace{1ex}
 \textit{\textbf{TABLE'S NOTES:}} \small{See \cite{Goodfellow, he2016deep} for more details about the different layers in a deep neural network. The compound $(Conv(5\times5@32,\,1,2){/}BN{/}ReLu{/}MaxPool(2\times2,\,1,0)))$ indicates a simple convolutional network (ConvNet) including a convolutional layer ($Conv$) using $32$ filters of size $5\times 5$, stride $1$, padding $2$, followed by a batch normalization layer ($BN$), a nonlinear activation function ($ReLu$) and, finally, a 2-D max-pooling layer with a channel of size $2\times2$, stride 1 and padding 0. The syntax $FC(C{/}Softmax)$ denotes a layer of $C$ fully connected neurons followed by the $softmax$ layer. Moreover, $(AvgPool)$, $(gAvg.Pool)$, and $(DropOut)$ refer to the 2D average-pooling, global average-pooling, and drop-out layers, respectively. The syntax $addition(1){/}ReLu$ indicates the existence of an \textit{identity shortcut} with functionality such that the output of a given block, say $B_1$ (or $B_2$ or $B_3$), is directly fed to the $addition$ layer and then to the ReLu layer while $addition(2){/}ReLu$ in a block shows the existence of a \textit{projection shortcut} with functionality such that the output from the two first ConvNets is added to the output of the third ConvNet and then the output is passed through the $ReLu$ layer. An open-source implementation of the \texttt{ResNet-20} and \texttt{LeNet-like} networks described above as components in Matlab programs of algorithms presented in \cite{yousefi2023deep} is available on \url{https://github.com/MATHinDL/sL_QN_TR/}.})
\end{minipage}}
\end{tabular}
\label{Nets}
\end{table}
\subsection{Classification problems}\label{body3b}

To solve an image classification problem for images with unknown classes/labels, we need to seek an optimal classification model by using a $C$-class training dataset $\{(x_i,y_i)_{i=1}^{N}\}$ with an image $x_i \in \mathbb{R}^{d}$ and its one-hot encoded label $y_i \in \mathbb{R}^C$. To this end, the generic DL problem \eqref{eq.problem} is minimized, where its single loss function $f_i = L(y_i, h(x_i; .))$ is \textit{softmax cross-entropy} as follows
\begin{equation}\label{single_loss_class}
	f_i(w) = - \sum_{k=1}^{C} (y_i)_k \log (h(x_i; w))_k,\quad i=1,\dots,N.
\end{equation}
In \eqref{single_loss_class}, the output $h(x_i; w)$ is a prediction provided by a DNN whose last layer includes the softmax function. For this classification task, we have considered two types of networks, the \texttt{LeNet-like} network with a shallow structure inspired by \texttt{LeNet-5} \cite{lecun1998gradient}, and the modern residual network \texttt{ResNet-20} \cite{he2016deep} with a deep structure. See \Cref{Nets} for the network's architectures. We have also considered the two most popular benchmarks; the \texttt{MNIST} dataset \cite{lecun1998mnist} with $7\times 10^4$ samples of handwritten gray-scale digit images of $28\times 28$ pixels and the \texttt{CIFAR10} dataset \cite{krizhevsky2009learning} with $6\times 10^4$ RGB images of $32\times32$ pixels, both in 10 categories. Every single image of \texttt{MNIST} and \texttt{CIFAR10} datasets is defined as a 3-D numeric array $x_i \in \mathbb{R}^d$ where $d = {28 \times 28 \times 1}$ and $d = {32 \times 32 \times 3}$, respectively. Moreover, every single label $y_i$ must be converted into a one-hot encoded label as $y_i \in \mathbb{R}^C$, where $C=10$. In both datasets, $\hat{N}=10^4$ images are set aside as testing sets, and the remaining $N$ images are set as training sets. Besides the zero-one rescaling, we implemented z-score normalization to have zero mean and unit variance. Precisely, all $N$ training images undergo normalization by subtracting the mean (an $h \times w \times c$ array) and dividing by the standard deviation (an $h \times w \times c$ array) of training images as an array sized $h \times w \times c \times N$. Here, $h$, $w$, and $c$ denote the height, width, and number of channels of the images, respectively, while $N$ represents the total number of images. Test data are also normalized using the same parameters as in the training data.\footnote{Mathematically, we have already denoted $i$th image of dimension $d$ as $x_i\in\mathbb{R}^d$ where $d = h\times w\times c$.} 

\Cref{Fig_Acc_Mnist_magnified,Fig_Acc_Cifar_magnified,Fig_Loss_Mnist_magnified,Fig_Loss_Cifar_magnified} 
show the variations, the mean and standard error obtained from 5 separate runs, of the aforementioned measures for both train and test datasets of ASNTR with $C_2=1, 10^2,10^8$ in $\rho_{\D_k}$, and $C_1=1$ in $\rho_{\N_k}$ within fixed budgets of gradient evaluations $N_g^{\text{max}}=6\times 10^5$ for \texttt{MNIST} and $N_g^{\text{max}}=9\times 10^6$ for \texttt{CIFAR10}. These figures demonstrate that ASNTR achieves higher training and testing accuracy than STORM in all considered values of $C_2$ except in Fig.~\ref{Fig_Acc_Cifar_magnified} with $C_2=1$ by which ASNTR can be comparable with STORM. Nevertheless, ASNTR is capable of achieving higher 
\begin{figure}[t]
    \centering
    \includegraphics[width=3.5cm, height=5.7cm]{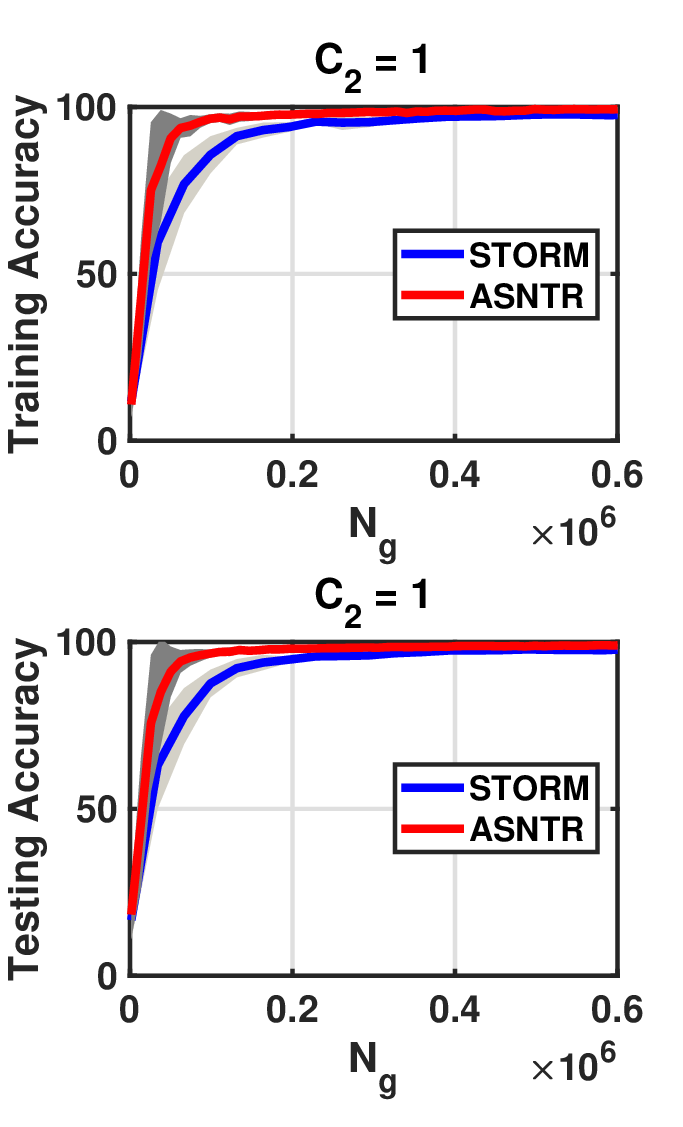}
    \includegraphics[width=3.5cm, height=5.7cm]{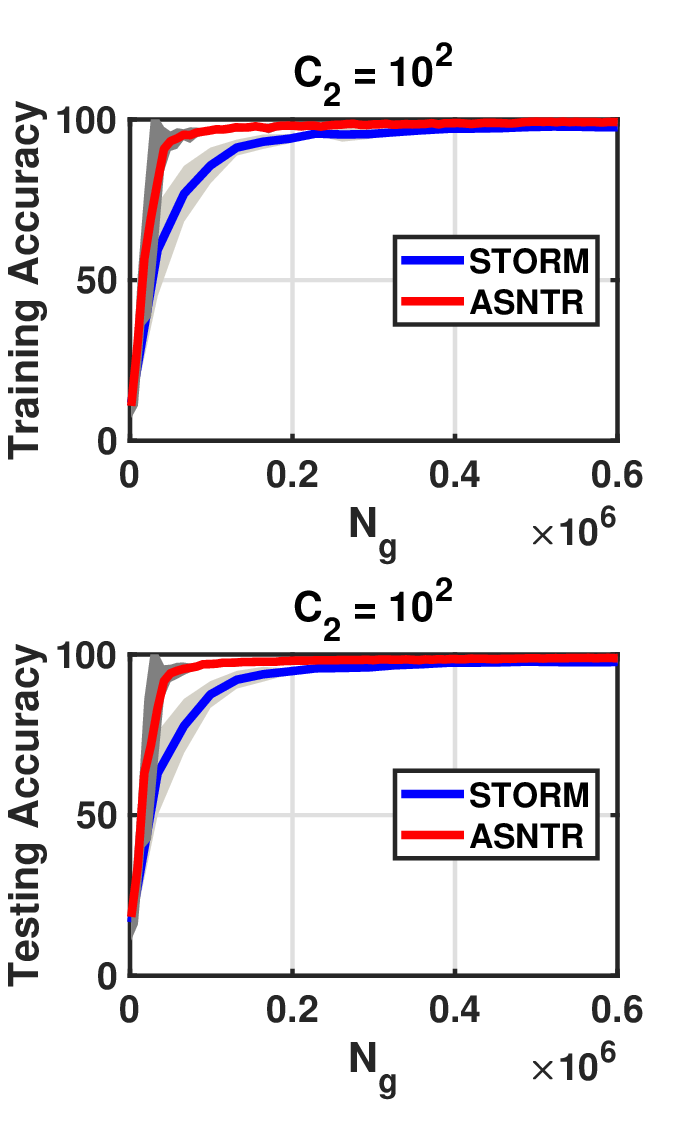}
    \includegraphics[width=3.5cm, height=5.7cm]{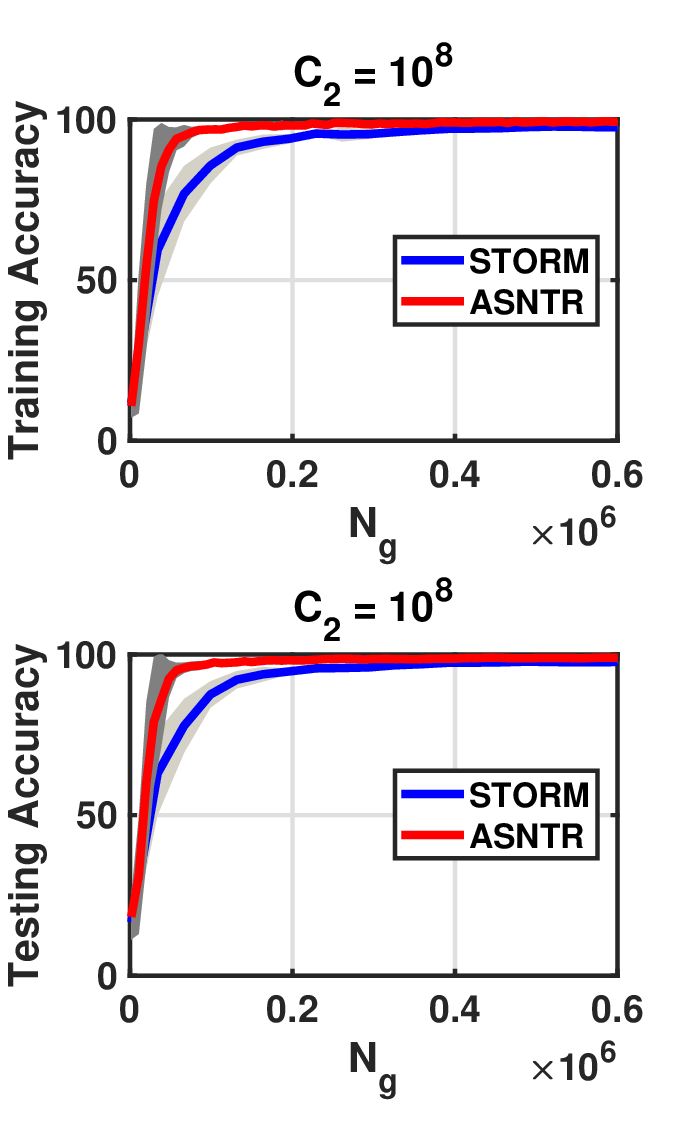}
    \caption{\small The accuracy variations of STORM and ASNTR on \texttt{MNIST}.}
    \label{Fig_Acc_Mnist_magnified}
\end{figure}
\vspace{-45pt}
\begin{figure}[H]
    \centering
    \includegraphics[width=3.5cm, height=5.7cm]{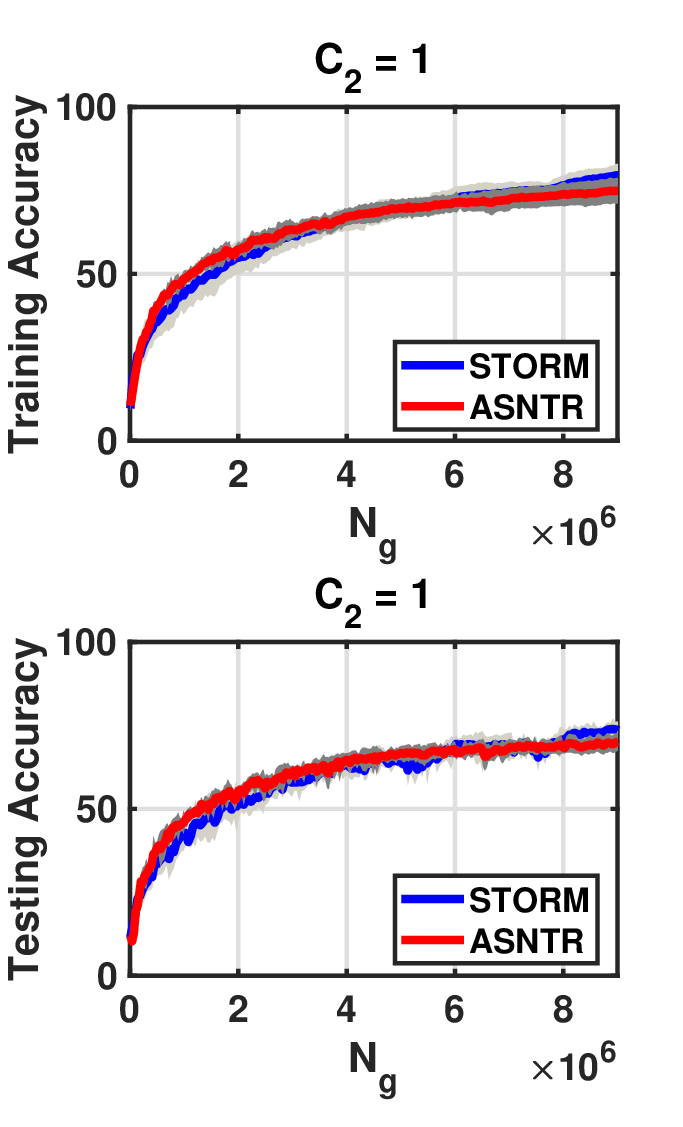}
    \includegraphics[width=3.5cm, height=5.7cm]{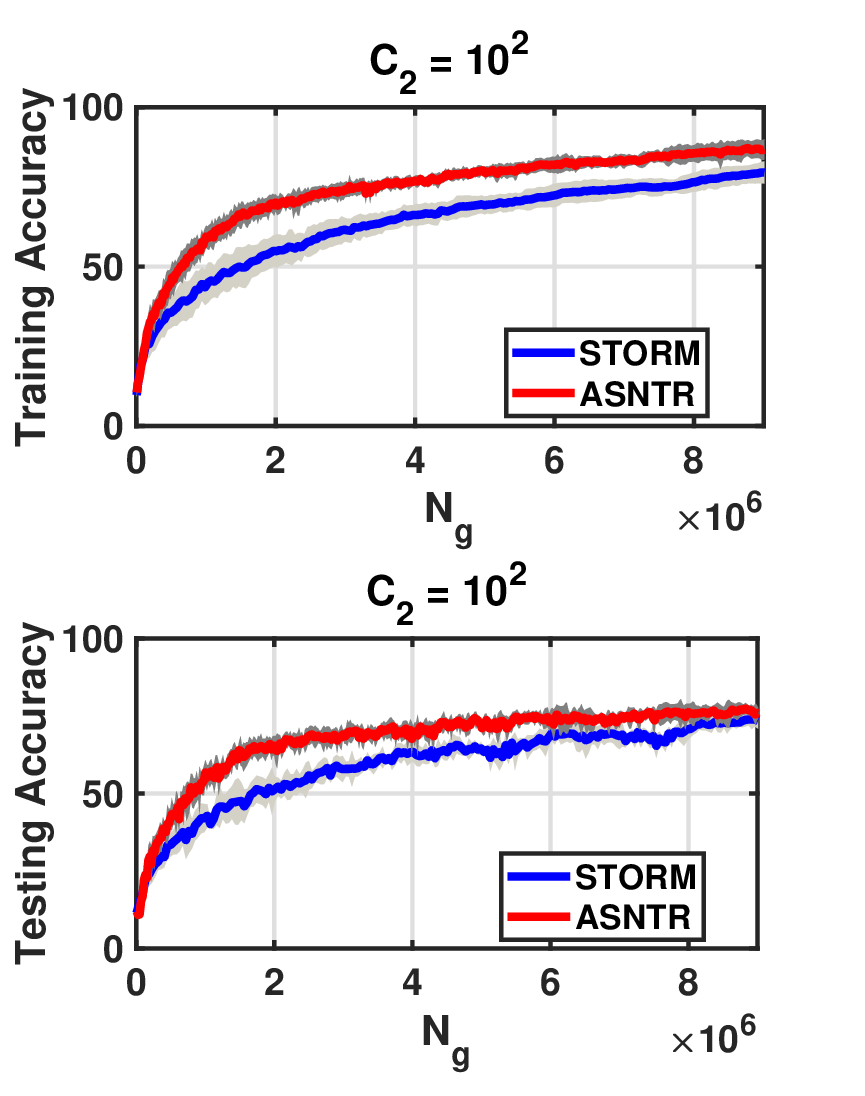} 
    \includegraphics[width=3.5cm, height=5.7cm]{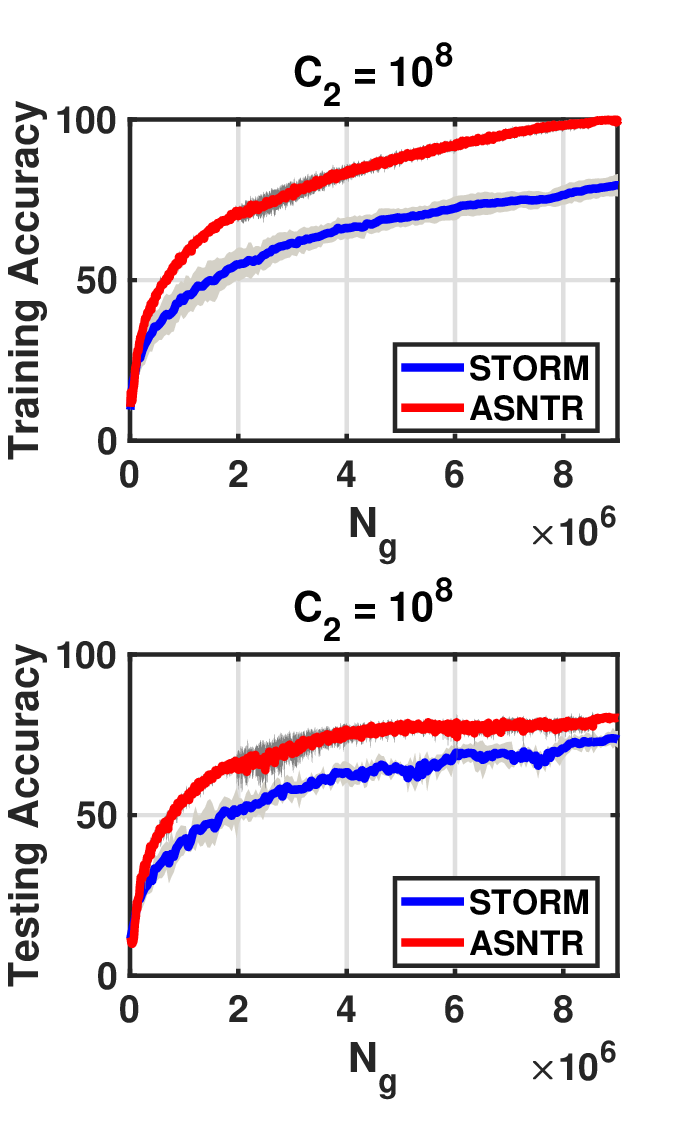}
    \caption{\small The accuracy variations of STORM and ASNTR on \texttt{CIFAR10}.}
    \label{Fig_Acc_Cifar_magnified}
\end{figure}
\begin{figure}[t]
    \centering
    \includegraphics[width=3.5cm, height=5.75cm]{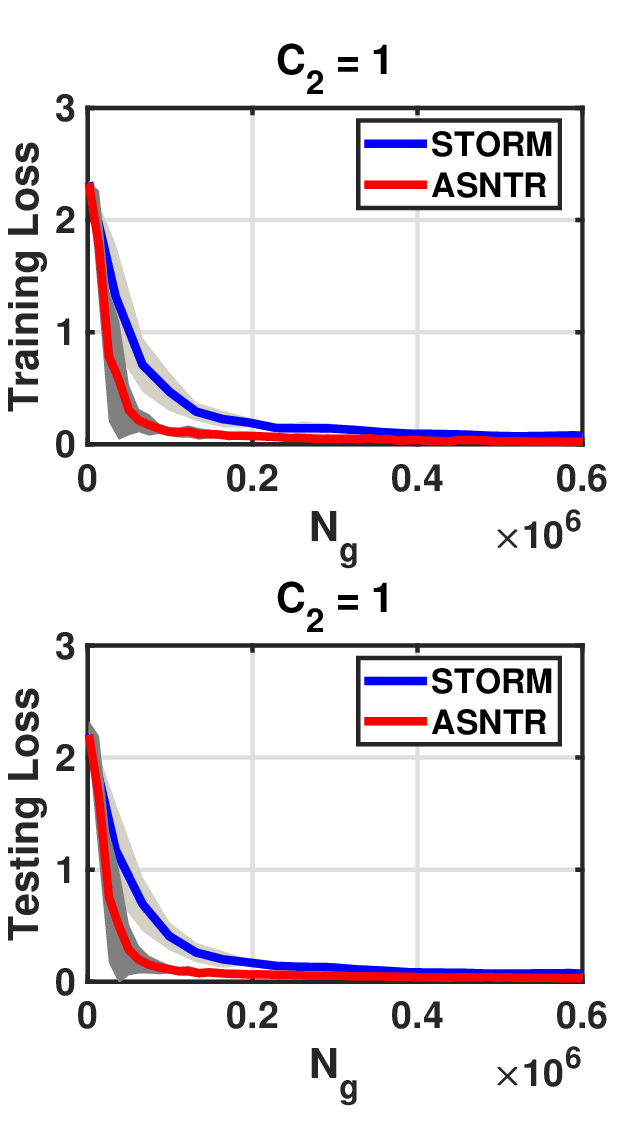}
    \includegraphics[width=3.5cm, height=5.75cm]{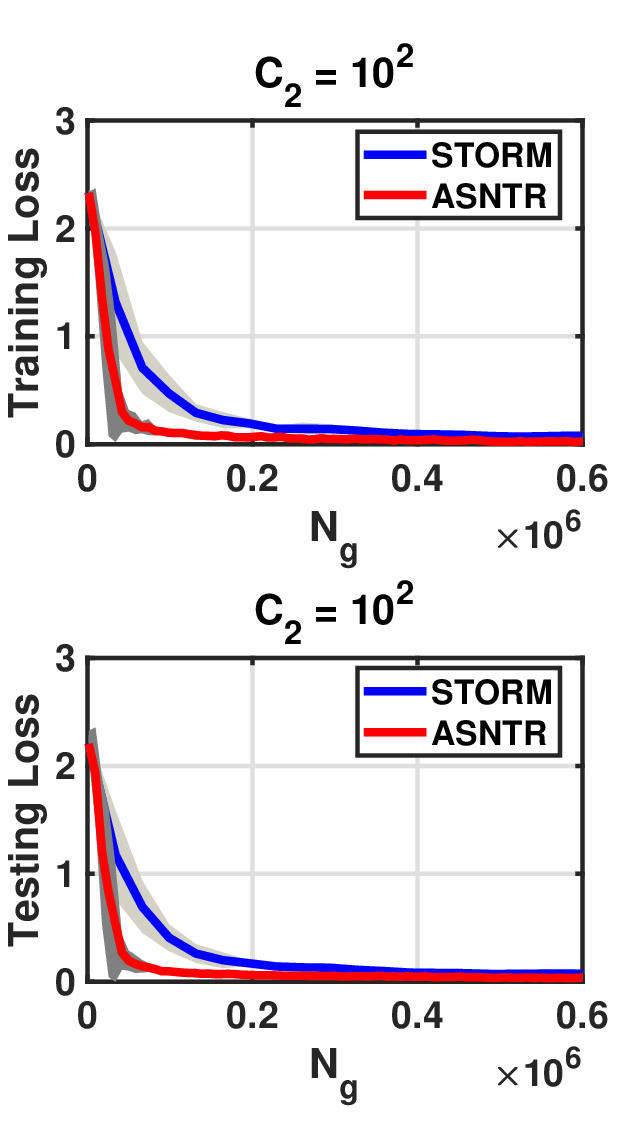}
    \includegraphics[width=3.5cm, height=5.75cm]{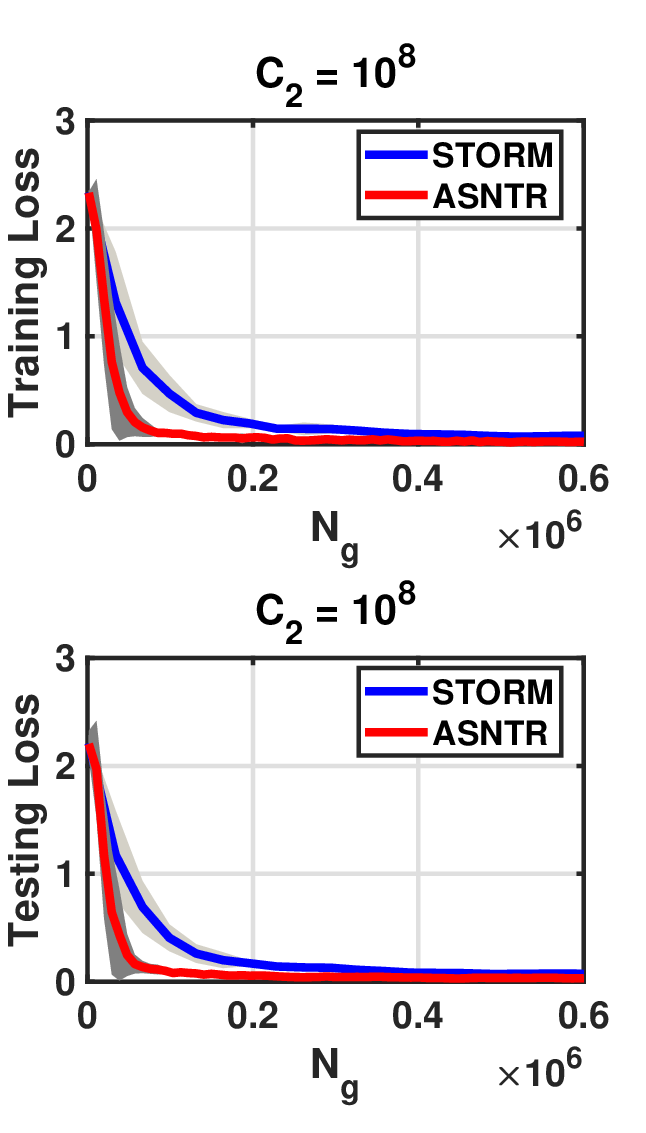}
    \caption{\small The loss variation of STORM and ASNTR on \texttt{MNIST}.}
    \label{Fig_Loss_Mnist_magnified}
\end{figure}

\noindent accuracy (lower loss) with fewer $N_g$ in comparison with STORM. Moreover, these figures indicate the robustness of ASNTR for larger values of $C_2$ meaning higher rates of non-monotonicity.
\subsection{Regression problem}\label{body3c} 

This example shows how to fit a regression model using a neural network to be able to predict the angles of rotation of handwritten digits which is useful for optical character recognition. To find an optimal regression model, the generic problem \eqref{eq.problem} is minimized, where $f_i = L(y_i, h(x_i; .))$ with a predicted output $h(x_i; w)$ is \textit{half-mean-squared error} as follows
\begin{equation}\label{single_loss_reg}
	f_i(w) = - \frac{1}{2} ( y_i - h(x_i; w) )^2,\quad i=1,\dots,N.
\end{equation}
In this regression example, we have considered a convolutional neural network (CNN) with an architecture named \texttt{CNN-Rn} as indicated in \Cref{Nets}. We have also used the \texttt{DIGITS} dataset containing $10\times 10^3$ synthetic images with $28\times 28$ pixels as well as their angles (in degrees) by which each image is rotated. Every single image is defined as a 3-D numeric array $x_i \in \mathbb{R}^d$ where $d = {28 \times 28 \times 1}$. Moreover, the response $y_i$ (the rotation angle in degrees) is approximately uniformly distributed between $-45^{\circ}$ and $45^{\circ}$. Each training and testing dataset has the same number of images ($N=\hat{N}=5\times 10^3$). Besides zero-one rescaling, we have also applied zero-center normalization to have zero mean. The problem \eqref{eq.problem} with single loss function \eqref{single_loss_reg} is solved for $w \in \mathbb{R}^n$ where $n=16,881$ using \texttt{CNN-Rn} for \texttt{DIGITS}. In this experiment, the accuracy is the number of predictions in percentage within an acceptable error margin (threshold) that we have set to be 10 degrees.
\begin{figure}[t]
    \centering
    \includegraphics[width=3.5cm, height=5.75cm]{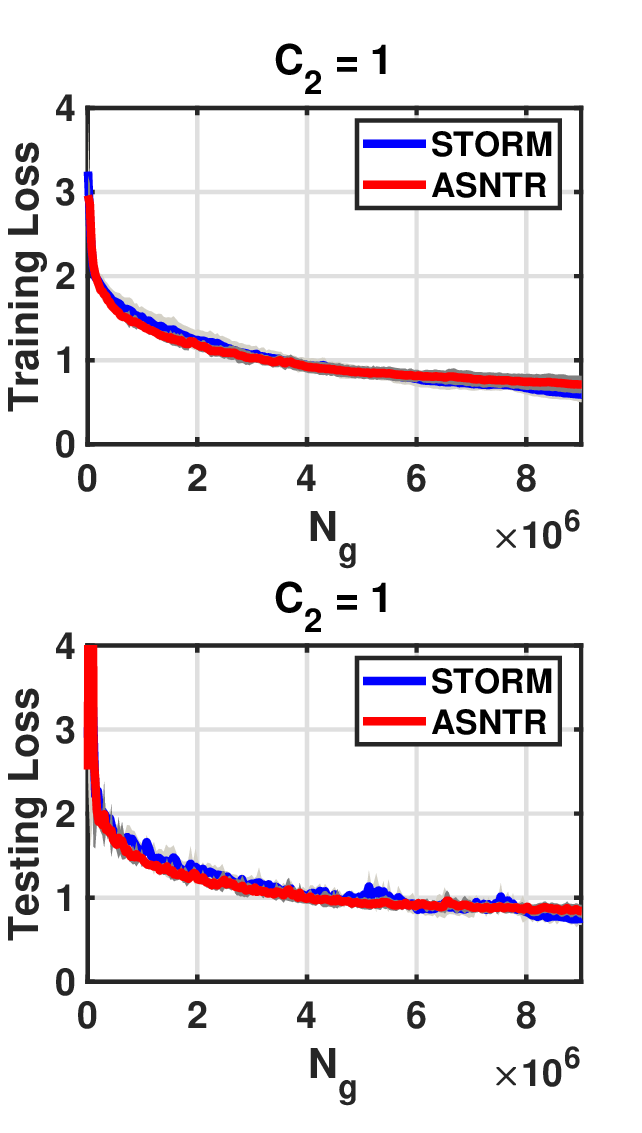}
    \includegraphics[width=3.5cm, height=5.75cm]{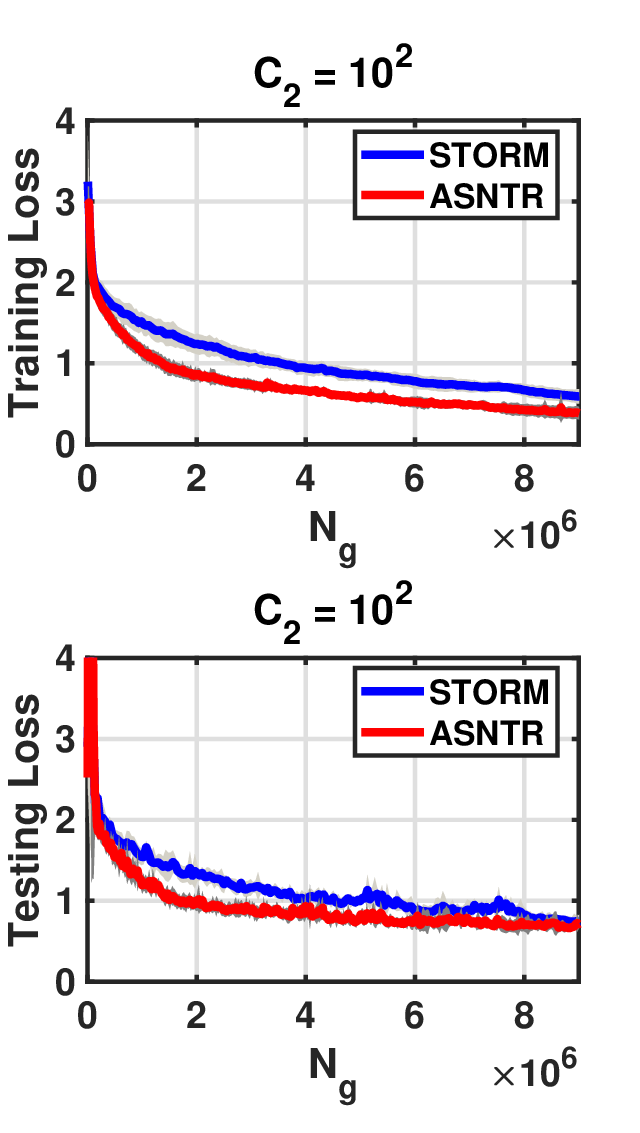}
    \includegraphics[width=3.5cm, height=5.75cm]{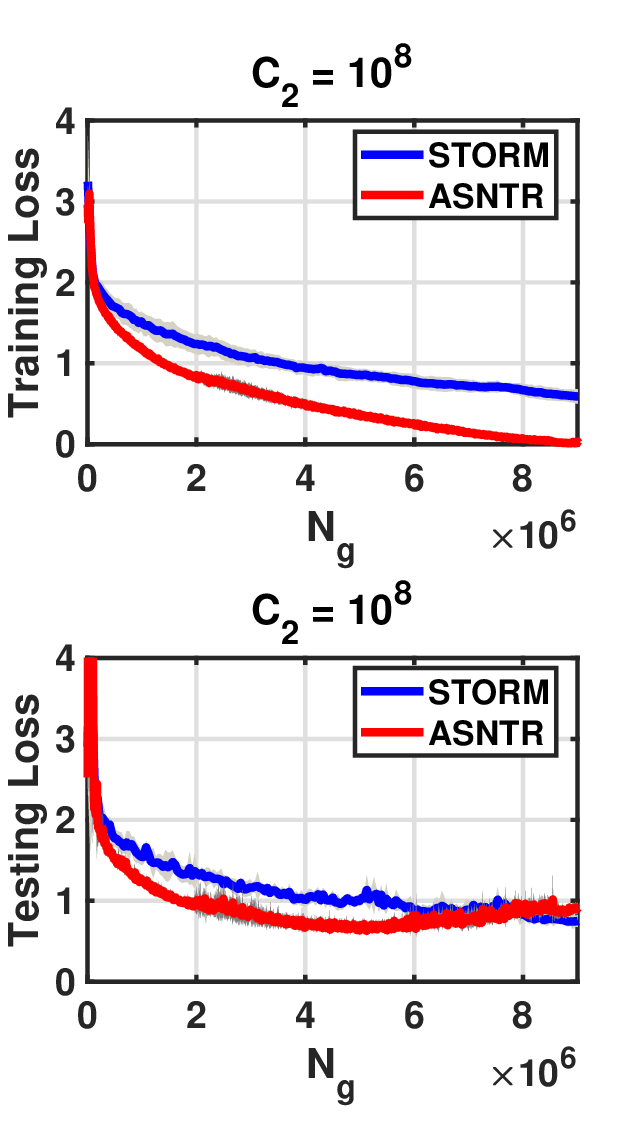}
    \caption{\small The loss variation of STORM and ASNTR on \texttt{CIFAR10}.}
    \label{Fig_Loss_Cifar_magnified}
\end{figure}
\vspace{-45pt}
\begin{figure}[H]
    \centering
    \includegraphics[width=3.5cm, height=5.7cm]{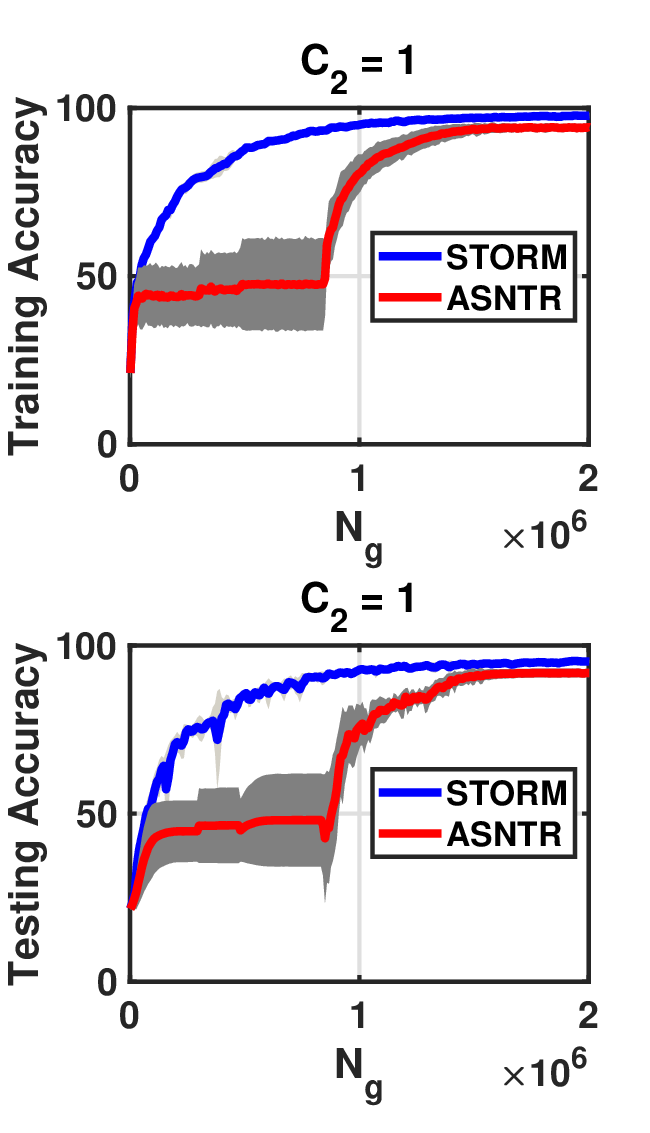}
    \includegraphics[width=3.5cm, height=5.7cm]{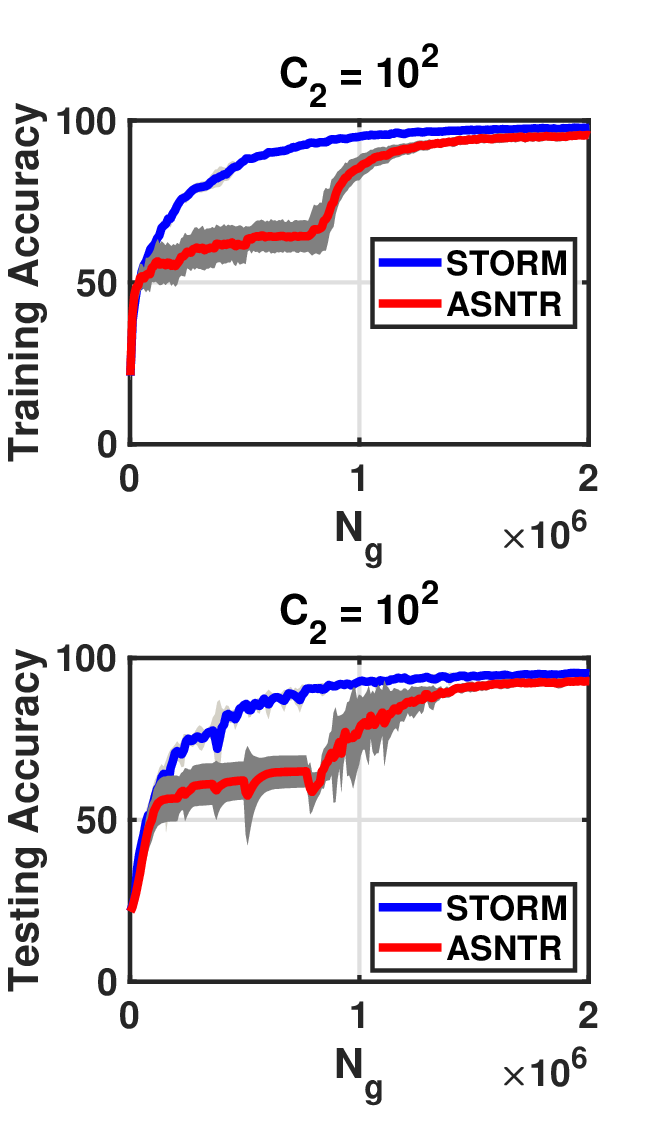}
    \includegraphics[width=3.5cm, height=5.7cm]{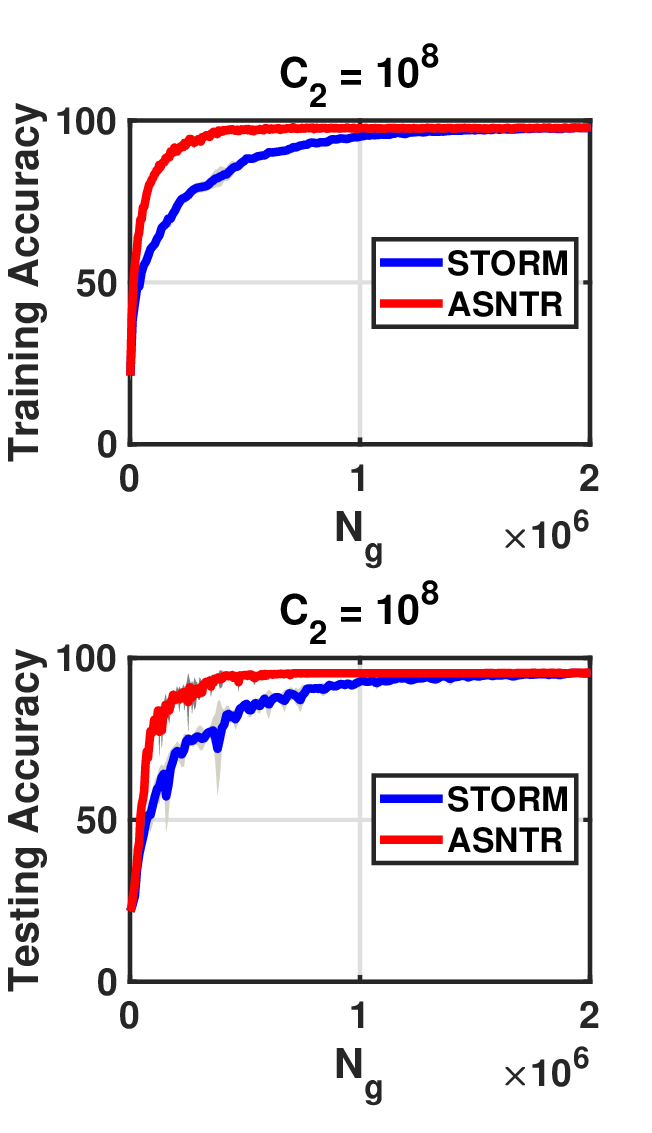}
    \caption{\small The accuracy variations of STORM and ASNTR on \texttt{DIGITS}.}
    \label{Fig_Acc_Digits_magnified}
\end{figure}
\vspace{-45pt}
\begin{figure}[H]
    \centering
    \includegraphics[width=3.5cm, height=5.75cm]{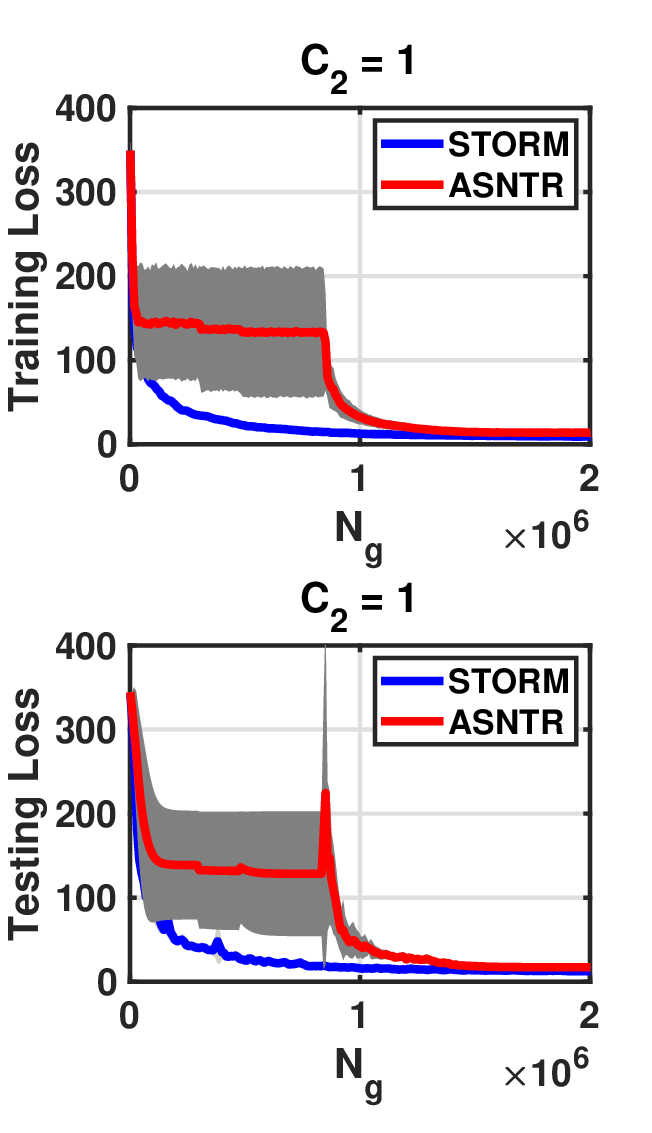}
    \includegraphics[width=3.5cm, height=5.75cm]{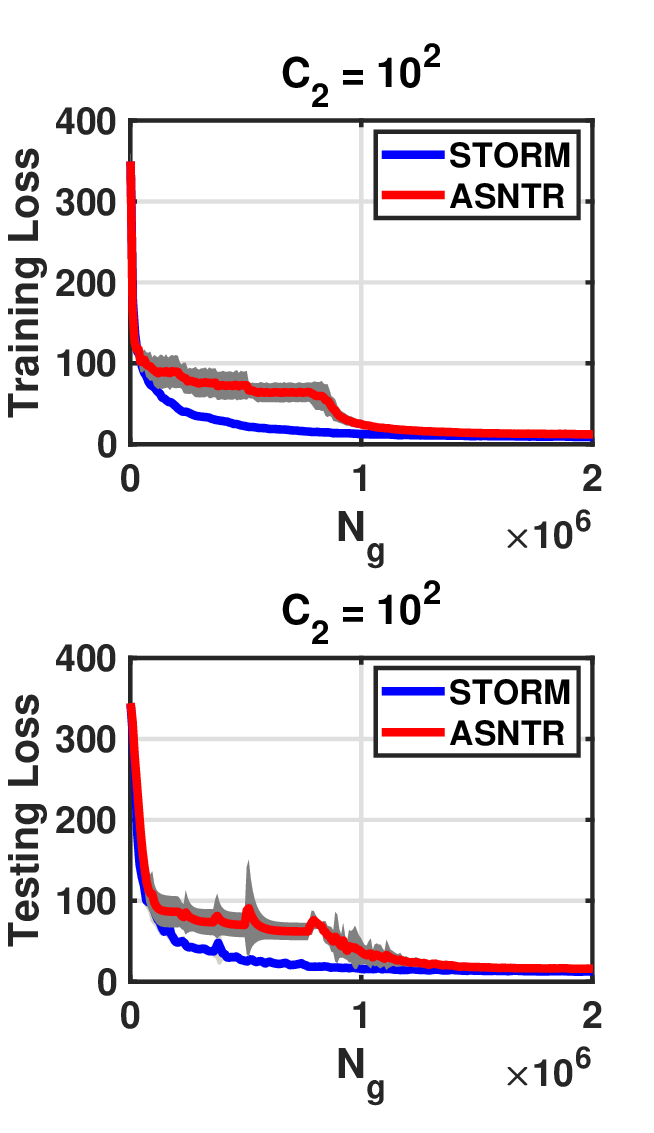}
    \includegraphics[width=3.5cm, height=5.75cm]{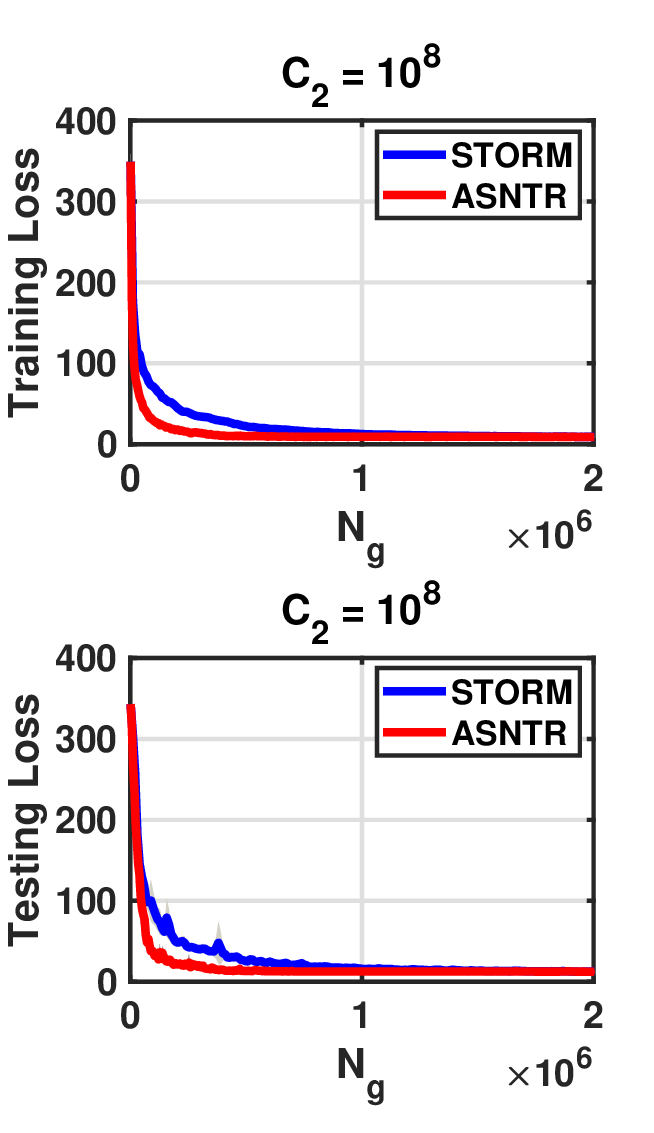}
    \caption{\small The loss variation of STORM and ASNTR on \texttt{DIGITS}.}
    \label{Fig_Loss_Digits_magnified}
\end{figure}

\Cref{Fig_Acc_Digits_magnified,Fig_Loss_Digits_magnified} show training accuracy and loss, and testing accuracy and loss variations of ASNTR for \texttt{DIGITS} dataset with three values of $C_2$ within a fixed budget of gradient evaluations $N_g^{\text{max}}=2\times 10^6$. These figures also illuminate how resilient ASNTR is for the highest value of $C_2$. Despite several challenges in the early stages of the training phase with $C_2=1$ and $C_2=10^2$, ASNTR can overcome them and achieve accuracy levels comparable to those of the STORM algorithm.

\subsection{Additional results}\label{body3d}
We present two additional figures (Figs.~\ref{Fig_Scatter_magnified} and \ref{Fig_bs_magnified}) containing further details regarding our proposed algorithm, ASNTR, with $C_2=1,\, 10^{2},\, 10^{8}$ in $\rho_{\mathcal{D}_k}$ and $C_1=1$ in $\rho_{\mathcal{N}_k}$. More specifically, these results aim to give useful insights concerning the sampling behavior of ASNTR.
Let S1 and S2 indicate the iterations of ASNTR at which steps 7 and 10, respectively, are executed using the increasing sampling rule \eqref{eq.adSampling}. 
When the 
\begin{figure}[H]
\centering
\begin{subfigure}{\textwidth}
\vspace{0.03cm}
\centering
{\includegraphics[width=4.2cm, height=4.5cm]{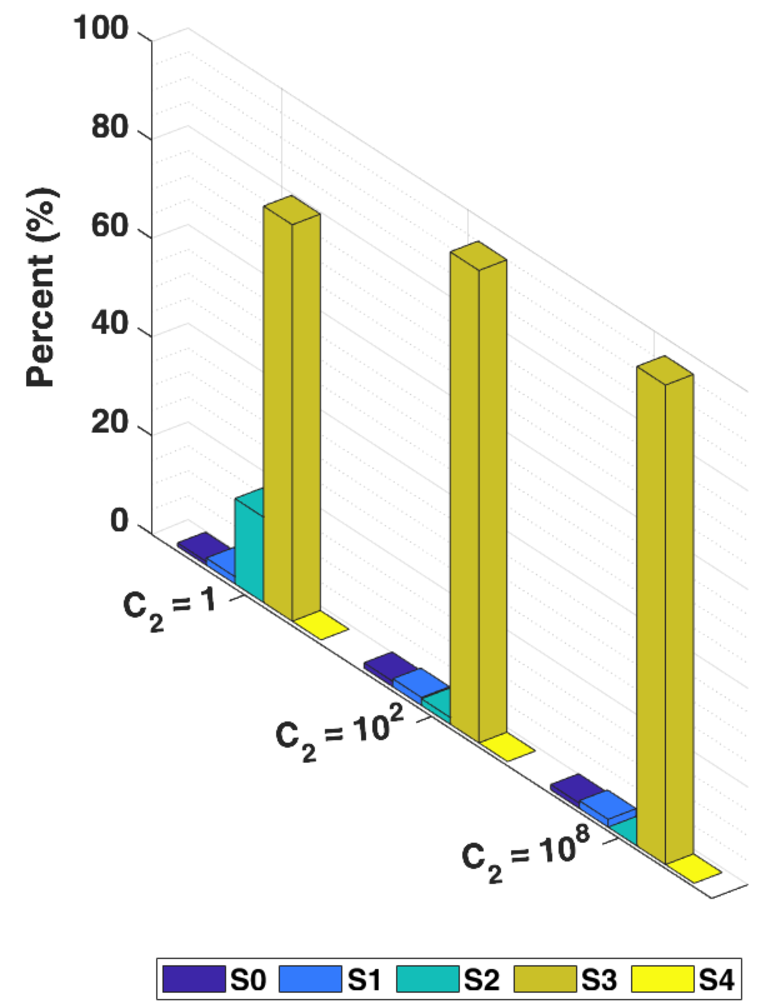}}
{\includegraphics[width=4.2cm, height=4.5cm]{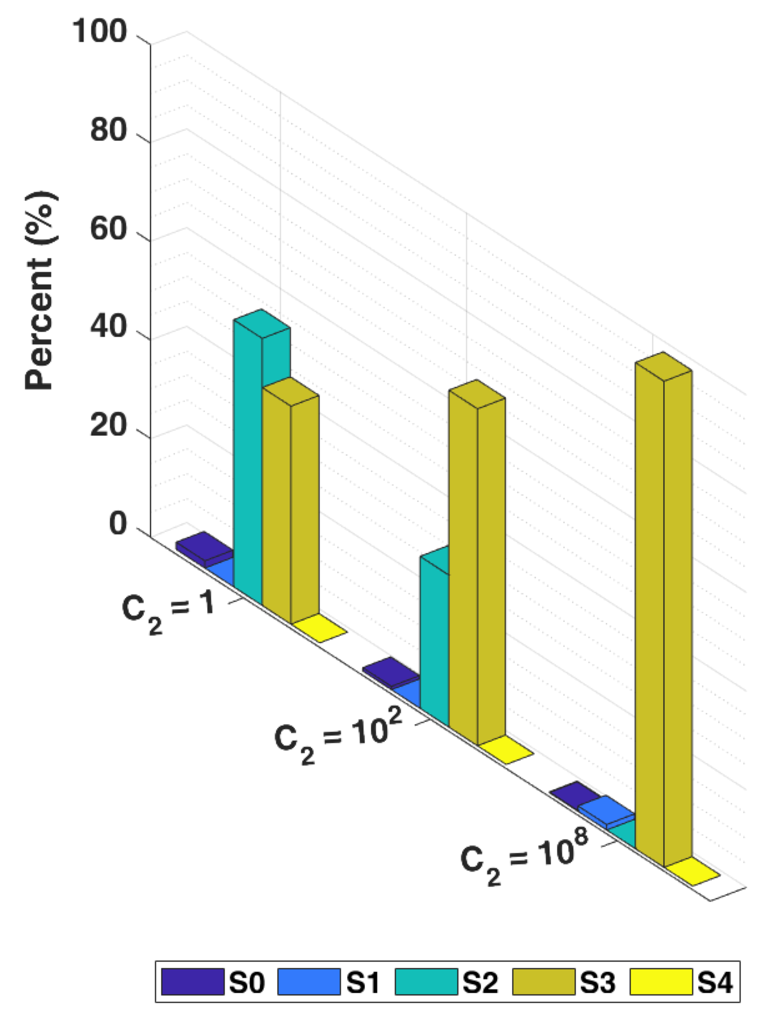}}
{\includegraphics[width=4.2cm, height=4.5cm]{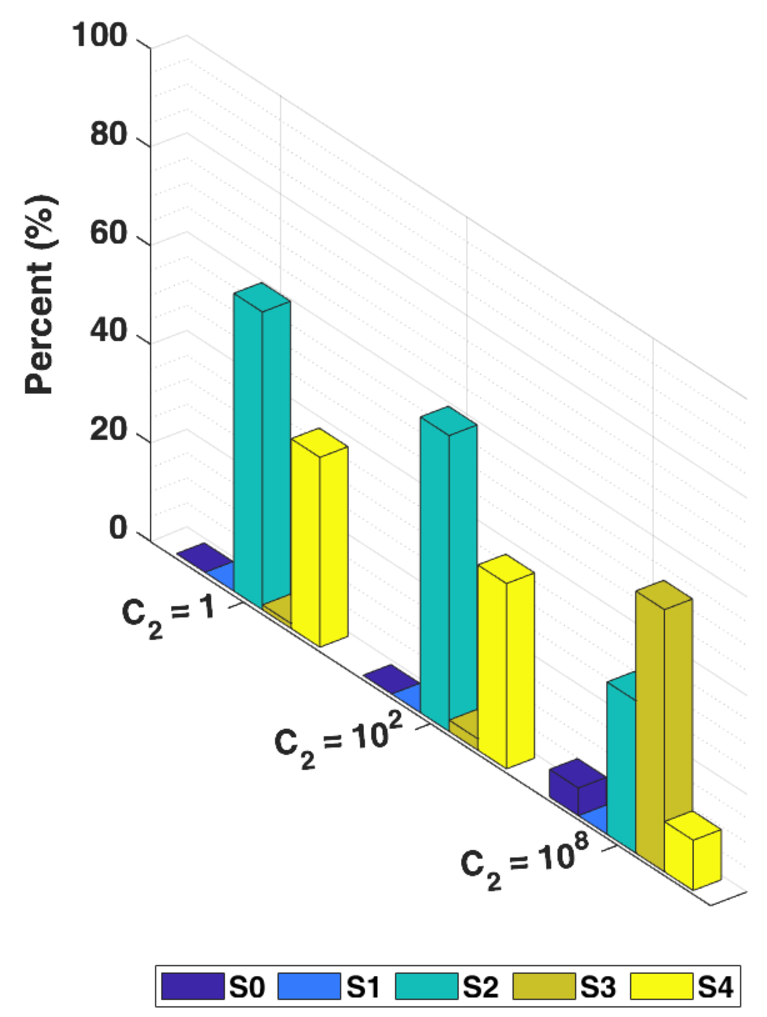}}
\vspace{-0.2cm}
\caption{\small Sampling behaviour of ASNTR on \texttt{MNIST} (left), \texttt{CIFAR10} (middle), \texttt{DIGITS} (right)} {\label{fig.Portion}}
\end{subfigure}
\begin{subfigure}{\textwidth}
\vspace{0.01cm}
\centering
{\includegraphics[width=13cm,height=3.5cm]{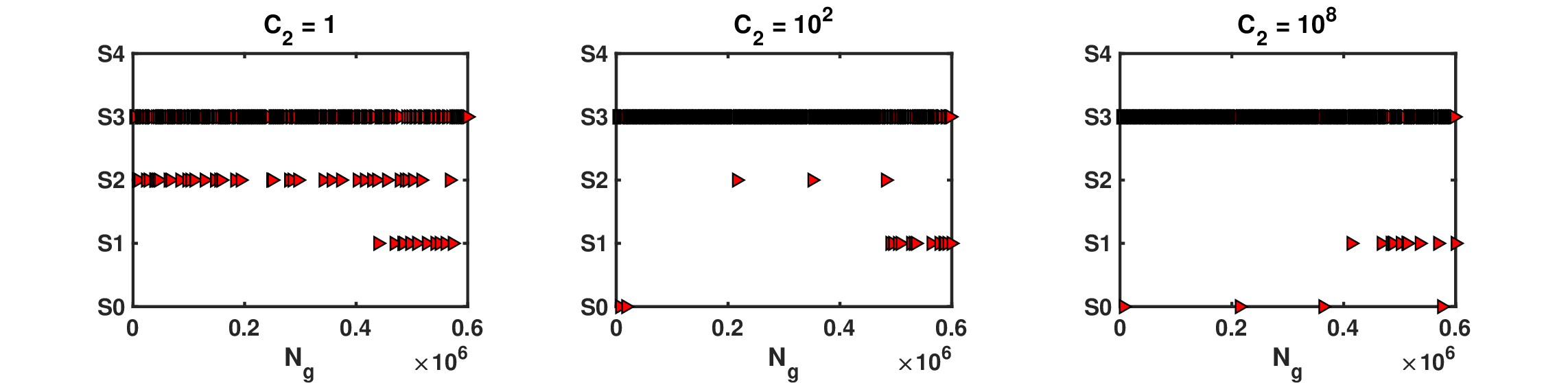}}
\vspace{-0.2cm}
\caption{\small Sampling behaviour with initial $\texttt{rng(42)}$ on \texttt{MNIST}}{\label{fig.ScatterMnist}}
\end{subfigure}
\begin{subfigure}{\textwidth}
\vspace{0.01cm}
\centering
{\includegraphics[width=13cm,height=3.5cm]{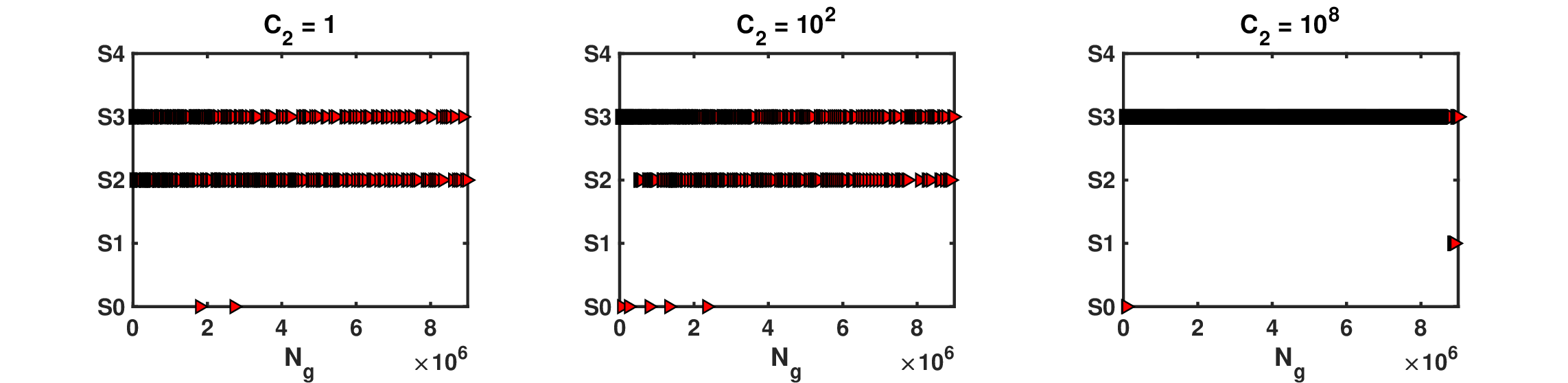}}
\vspace{-0.2cm}
\caption{\small Sampling behaviour with initial $\texttt{rng(42)}$ on \texttt{CIFAR10}}{\label{fig.ScatterCifar}}
\end{subfigure}
\begin{subfigure}{\textwidth}
\vspace{0.01cm}
\centering
{\includegraphics[width=13cm,height=3.5cm]{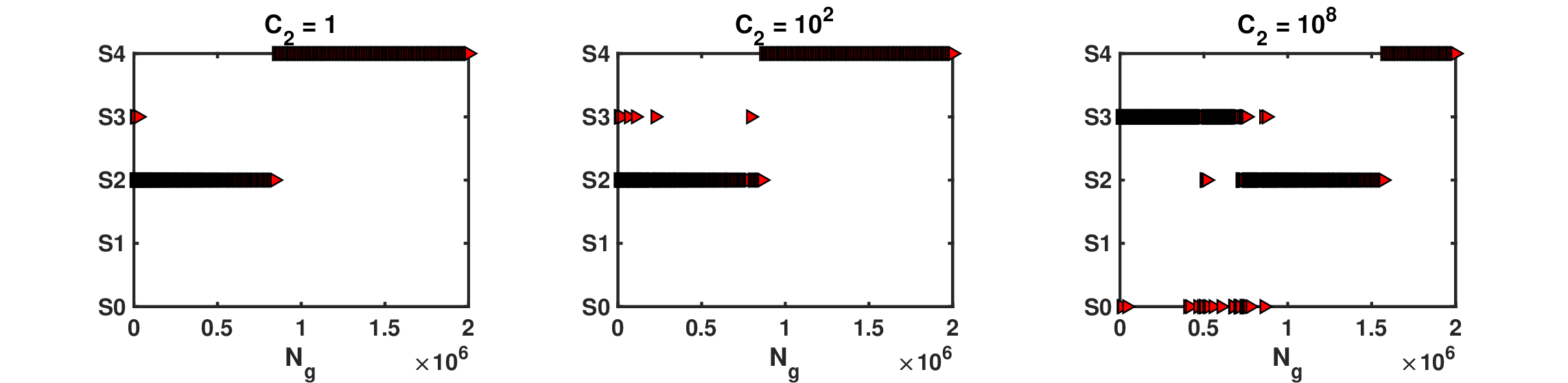}}
\vspace{0.01cm}
\caption{\small Sampling behaviour with initial $\texttt{rng(42)}$ on \texttt{DIGITS}}{\label{fig.ScatterDigits}}
\end{subfigure}
\caption{\small Tracking subsampling in ASNTR.}
\label{Fig_Scatter_magnified}
\end{figure}
\noindent cardinality of the sample set is not changed,
i.e., $N_{k+1}=N_k$, let S3 and S0 show the iterations at which new
samples through step 15 and current samples through 
step 13 are selected. We also define variable S4 representing the iteration of ASNTR at which all available samples ($N_k=N$) are needed for computing the required quantities. 
\begin{figure}[t]
\centering
\begin{subfigure}{\textwidth}
\vspace{0.1cm}
\centering
{\includegraphics[width=13cm,height=3.5cm]{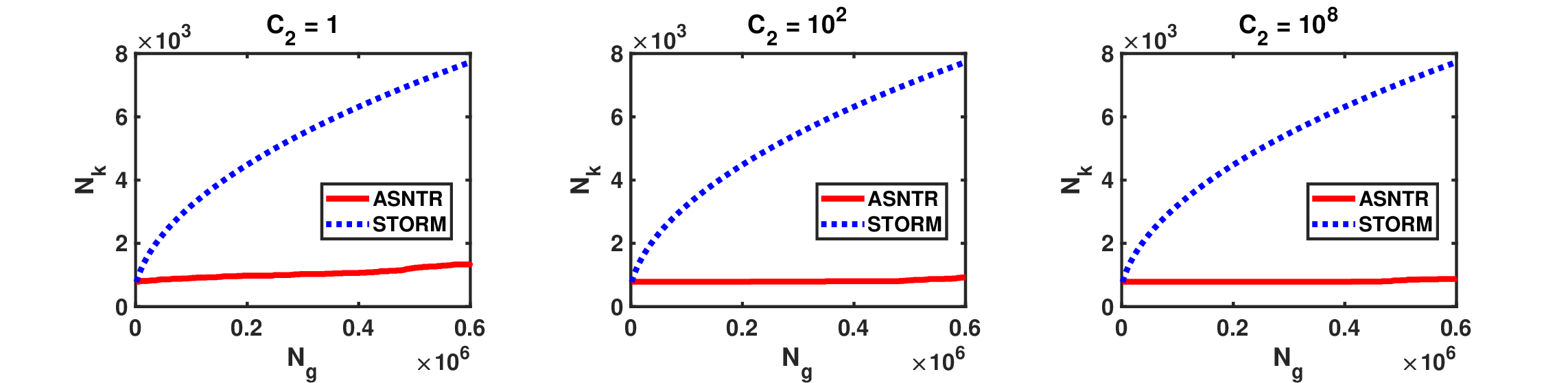}}
\vspace{0.1cm}
\caption{\small  MNIST ($N=6\times 10^4$)}
\end{subfigure}
\begin{subfigure}{\textwidth}
\vspace{0.1cm}
\centering
{\includegraphics[width=13cm,height=3.5cm]{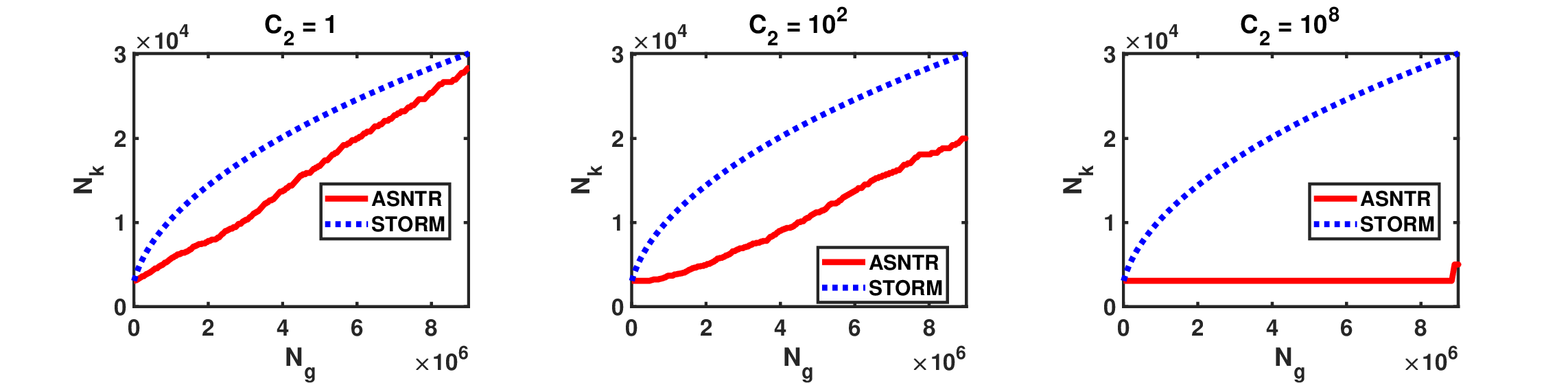}}
\vspace{0.1cm}
\caption{\small  \texttt{CIFAR10} ($N=5\times 10^4$)}
\end{subfigure}
\begin{subfigure}{\textwidth}
\vspace{0.1cm}
\centering
{\includegraphics[width=13cm,height=3.5cm]{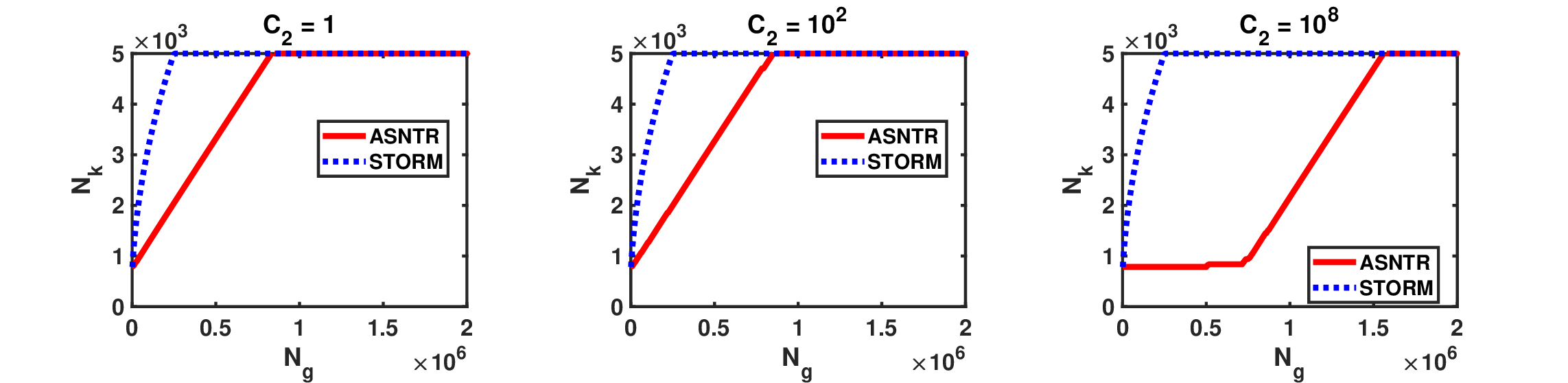}}
\vspace{0.1cm}
\caption{\small \texttt{DIGITS} ($N=5\times 10^3$)}
\end{subfigure}
\caption{\small Batch size progress with initial $\texttt{rng(42)}$.}
\label{Fig_bs_magnified}
\end{figure}

 \Cref{fig.Portion} shows the (average) contributions of the aforementioned sampling types in ASNTR running with five different initial random generators for \texttt{MNIST}, \texttt{CIFAR10}, and \texttt{DIGITS} with respectively predetermined $N_g^{\text{max}}=0.6\times 10^{6},\, 9\times 10^{6}$ and $2\times 10^{6}$; however, considering only a specific initial seed, e.g. \texttt{rng(42)}, Figs.~\ref{fig.ScatterMnist}, \ref{fig.ScatterCifar} and \ref{fig.ScatterDigits} indicate when/where each of these sampling types is utilized in ASNTR. Obviously, the contribution rate of $S3$ is influenced by $S2$, where ASNTR has to increase the batch size if $\rho_{\D_k}<\nu$. In fact, the larger $C_2$ in $\rho_{D_k}$ results in the higher portion of $S3$ and the smaller portion of $S2$. Moreover, using a large value of $C_2$, there is no need to increase the current batch size unless the current iterate approaches a stationary point of the current approximate function. This, in turn, leads to increasing the portion of $S1$, which usually happens at the end of the training stage. In addition, we have also found that the value of $C_2$ in $\tilde{t}_k$ plays an important role in the robustness of our proposed algorithm as observed in Figs.~\ref{Fig_Acc_Mnist_magnified} to \ref{Fig_Loss_Digits_magnified}; in other words, the higher the  $C_2$, the more robust ASNTR is. These mentioned observations, more specifically, can be followed for every single dataset as below:
 \begin{itemize}
    \item \texttt{MNIST}: according to Fig.~\ref{fig.Portion}, the portion of the sampling type $S3$ is higher than the others. This means many new batches without an increase in the batch size are applied in ASNTR during training; i.e., the proposed algorithm can train a network with fewer samples, and thus fewer gradient evaluations. Nevertheless, ASNTR with different values of $C_2$ in $\rho_{\mathcal{D}_k}$ increases the size of the mini-batches in some of its iterations; see the portions of $S1$ and $S2$ in Fig.~\ref{fig.Portion} or their scatters in Fig.~\ref{fig.ScatterMnist}. We should note that the sampling type $S1$ occurs almost at the end of the training phase where the algorithm tends to be close to a stationary point of the current approximate function; Fig.~\ref{fig.ScatterMnist} shows this fact.
    \item \texttt{CIFAR10}: according to Fig.~\ref{fig.Portion}, the portion of the sampling type $S3$ is still high. Unlike \texttt{MNIST}, the sampling type $S1$ rarely occurs during the training phase. On the other hand, a high portion of the increase of the sample size through $S2$ may compensate for the lack of sufficiently accurate functions and gradients required in ASNTR. These points are also illustrated in Fig.~\ref{fig.ScatterCifar}, which shows how ASNTR successfully trained the \texttt{ResNet-20} model without frequently enlarging the sample sizes. For both the \texttt{MNIST} and \texttt{CIFAR10} problems, $S3$ as the predominant type corresponds to $C_2=10^{8}$.
    \item \texttt{DIGITS}: according to Fig.~\ref{fig.Portion}, we observe that the main sampling types are $S2$ and $S4$. As the portion of $S2$ increases, the portion of $S3$ decreases and the highest portion of $S3$ corresponds to the largest value of $C_2$. This pattern is similar to what is seen in the \texttt{MNIST} and \texttt{CIFAR10} datasets. However, in the case of \texttt{DIGITS}, the portion of $S4$ is higher. This higher portion of $S4$ in \texttt{DIGITS} may be attributed to the smaller number of samples in this dataset ($N=5000$), which causes ASNTR to quickly encompass all the samples after a few iterations. Notably, the sampling type $S4$ occurs towards the end of the training phase, as shown in Fig.~\ref{fig.ScatterDigits}.
\end{itemize} 

\Cref{Fig_bs_magnified} compares the progression of batch size growth in both ASNTR and STORM. In contrast to the STORM algorithm, ASNTR increases the batch size only when necessary, which can reduce the computational costs of gradient evaluations. This is considered a significant advantage of ASNTR over STORM. However, according to this figure, the proposed algorithm needs fewer samples than STORM during the initial phase of the training task, but it requires more samples toward the end. Nevertheless, we should notice that the increase in batch size that happened at the end of the training phase is determined either by $S1$ or by $S4$ (see Figs.~\ref{fig.ScatterMnist} and \ref{fig.ScatterDigits}). In our experiments, we have observed that ASNTR does not use very large $N_g^{\text{max}}$, as it typically achieves the required training accuracy. 

\subsection{Comparison with ADAM}\label{body3e}
We have compared the proposed stochastic second-order method (ASNTR) with a popular efficient first-order optimizer, i.e., Adaptive Moment Estimation (ADAM) \cite{kingma2017adam} used in DL. We have implemented ADAM using MATLAB built-in function \texttt{adamupdate} in customized training loops. It's widely recognized that this optimizer is highly sensitive to the value of its hyper-parameters including the learning rate, $\alpha_k$, the gradient decay factor, $\beta_1$, the squared gradient decay factor, $\beta_2$, a small constant to prevent division by zero, $\epsilon$, and batch size. Users should be aware of the hyper-parameter choices and invest time in tuning them using techniques such as grid search. In our experiments, we set $\beta_1 = 0.9$, $\beta_2 = 0.999$, and $\epsilon=10^{-8}$, respectively, and focus our tuning effort on learning rate and batch size. The specified choices for tuning the learning rate were $\alpha_k = \alpha$ with $\alpha$ taking values from the set $\{10^{-4}, 10^{-3}, 10^{-2}\}$, and the corresponding values for batch size were $N_k = bs$ where ${bs}$ varied within $\{128, 256, 784\}$ for both \texttt{MNIST} and \texttt{DIGITS}, and within $\{128, 256, 3072\}$ for \texttt{CIFAR10}.

The best hyper-parameters were those that yielded the highest testing accuracy within a fixed budget of gradient evaluations. In all experiments, the optimal performance with Adam was consistently achieved using $\alpha=10^{-3}$ and $bs=128$. These settings were employed with ADAM in Figs.~\ref{Fig_Acc_Loss_MNIST},\ref{Fig_Acc_Loss_CIFAR} and \ref{Fig_Acc_Loss_DIGITS} (first rows) for comparison purposes against ASNTR with $C_2=10^8$ demonstrated in Figs.~\ref{Fig_Acc_Mnist_magnified} to \ref{Fig_Loss_Digits_magnified}. As these figures show, the tuned ADAM could produce the highest accuracy and lowest loss with fewer gradient evaluations in comparison with ASNTR. The common observation of rapid initial improvement achieved by ADAM, followed by a drastic slowdown, is well understood in practice for some first-order methods (see, e.g., \cite{bottou2018optimization}). First-order methods such as ADAM are computationally less expensive per iteration compared to second-order methods such as ASNTR, as they only involve the computation of one gradient compared to two gradients. Moreover, as already mentioned the initial sample size for ASNTR was set as $N_0 = d+1$ where $d=784$ for both \texttt{MNIST} and \texttt{DIGITS} and $d=3072$ for \texttt{CIFAR10}. By initially employing such large batch sizes against an obtained optimal batch size for ADAM, i.e., $N_k=128$, there are only a small number of updated iterates ($w_k$) during training with both second-order methods within the fixed budget of gradient evaluations. Therefore, allowing ASNTR to train networks within a larger number of gradient evaluations helps it to eventually achieve a higher level of accuracy while this cannot help ADAM. Note that we have performed this analysis between the tuned ADAM ($N_k=128$, and $\alpha_k=10^{-3}$) and ASNTR, where tuning the hyper-parameters of ADAM incurs significant time costs. When ADAM is used with suboptimal hyper-parameters, its sensitivity becomes evident, as illustrated in Figs.~\ref{Fig_Acc_Loss_MNIST},\ref{Fig_Acc_Loss_CIFAR} and \ref{Fig_Acc_Loss_DIGITS} (second rows) where batch size $N_k=d+1$ and different learning rates are considered. As these figures show, for a challenging problem such as \texttt{CIFAR10} classification, ADAM may not achieve a lower loss value than ASNTR, even when employing an optimal learning rate ($\alpha=10^{-2}$). Neither for the other challenging problem \texttt{DIGITS} ADAM could produce higher testing Accuracy than ASNTR. All the results shown in the three figures were obtained using the same seed for the random number generator (\texttt{rng(42)}). Moreover, due to awkward oscillations in Loss and Accuracy obtained by ADAM in \texttt{CIFAR10} and \texttt{DIGITS} classification problems, we have imposed a filter using \texttt{movmean} MATLAB built-in function (moving mean with a window of length $c$) in Figs.~\ref{Fig_Acc_Loss_CIFAR} and \ref{Fig_Acc_Loss_DIGITS}; in our experiments $c=30$.

\begin{figure}[H]
    \centering
    \includegraphics[width=5.8cm, height=3.9cm]{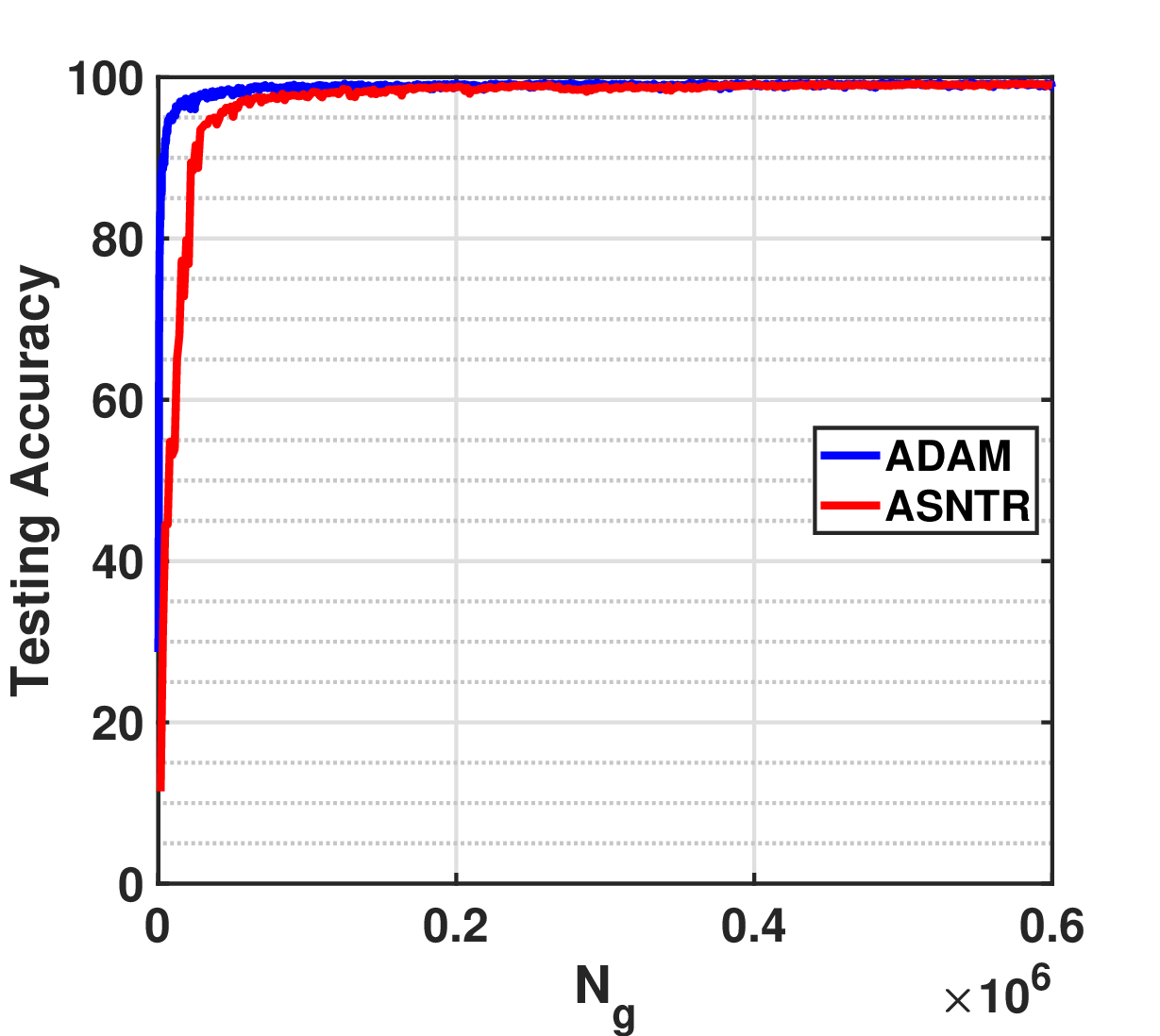}
    \includegraphics[width=5.8cm, height=3.9cm]{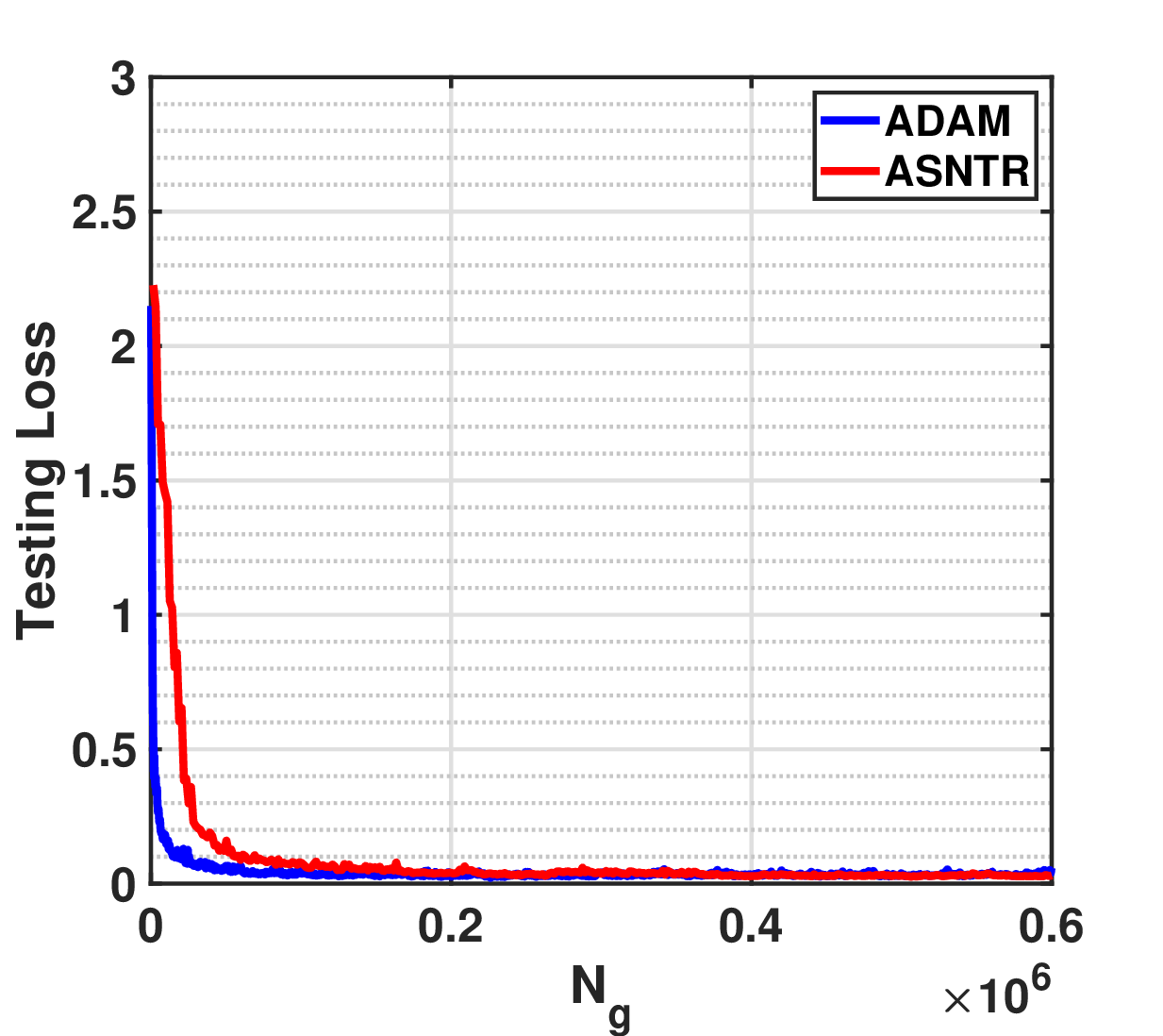}
    \includegraphics[width=5.8cm, height=3.9cm]{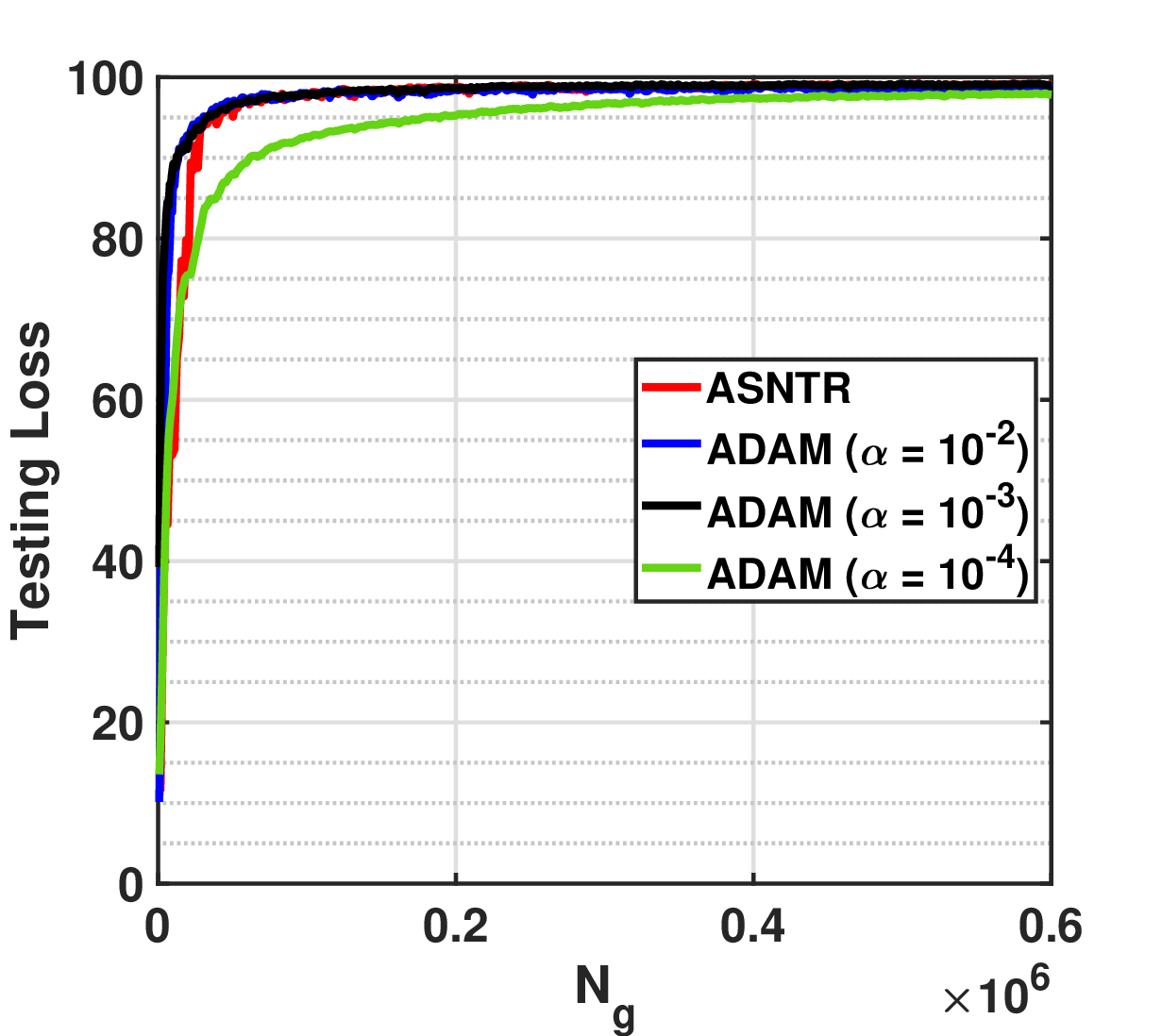}
    \includegraphics[width=5.8cm, height=3.9cm]{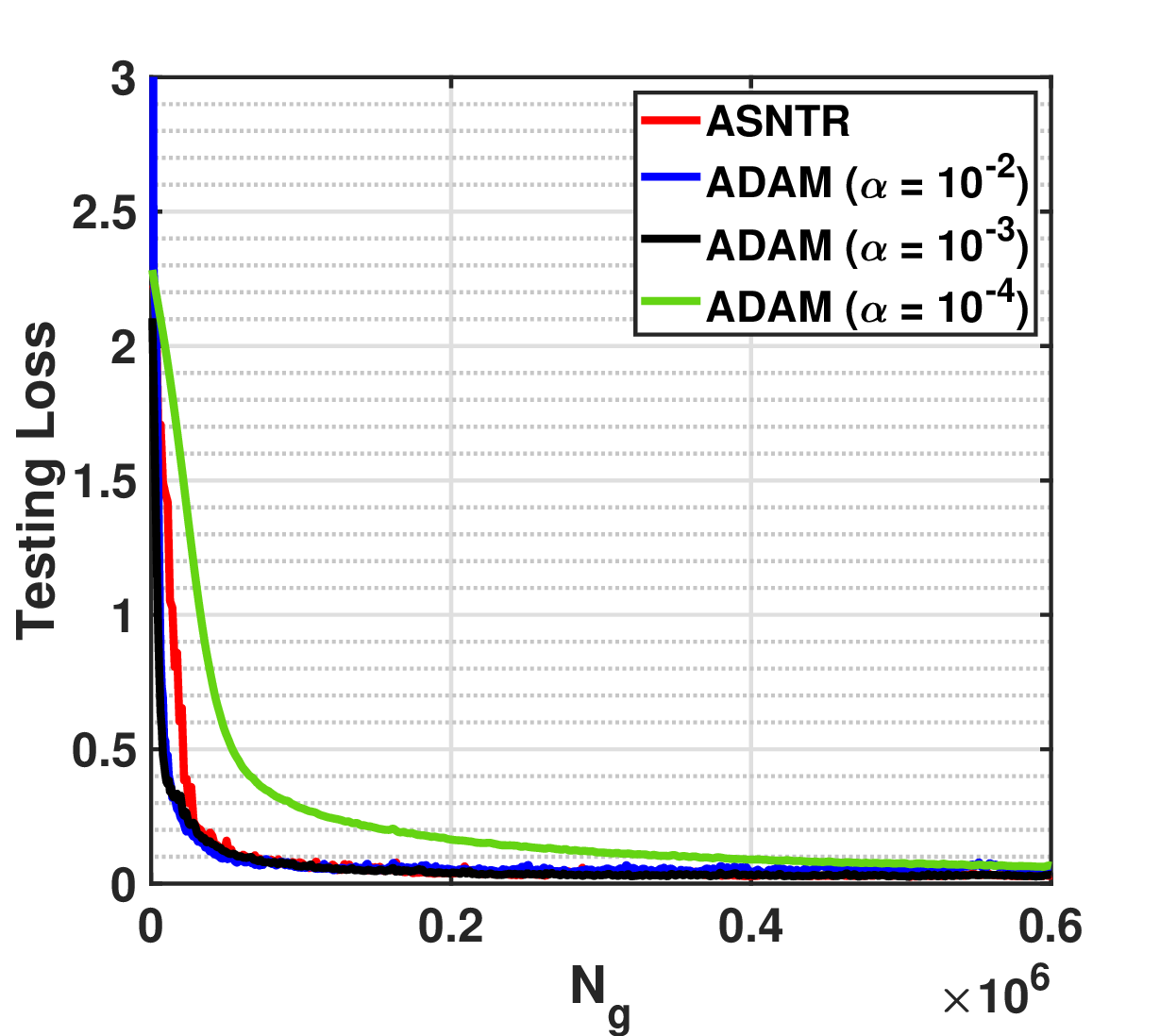}
    \caption{\small Comparison of ASNTR vs. tuned ADAM (first row), and vs. untuned ADAM (second row) on \texttt{MNIST} for training \texttt{LeNet-like} with \texttt{rng(42)}.}
    \label{Fig_Acc_Loss_MNIST}
\end{figure}

\begin{figure}[t]
    \centering
    \includegraphics[width=5.8cm, height=3.9cm]{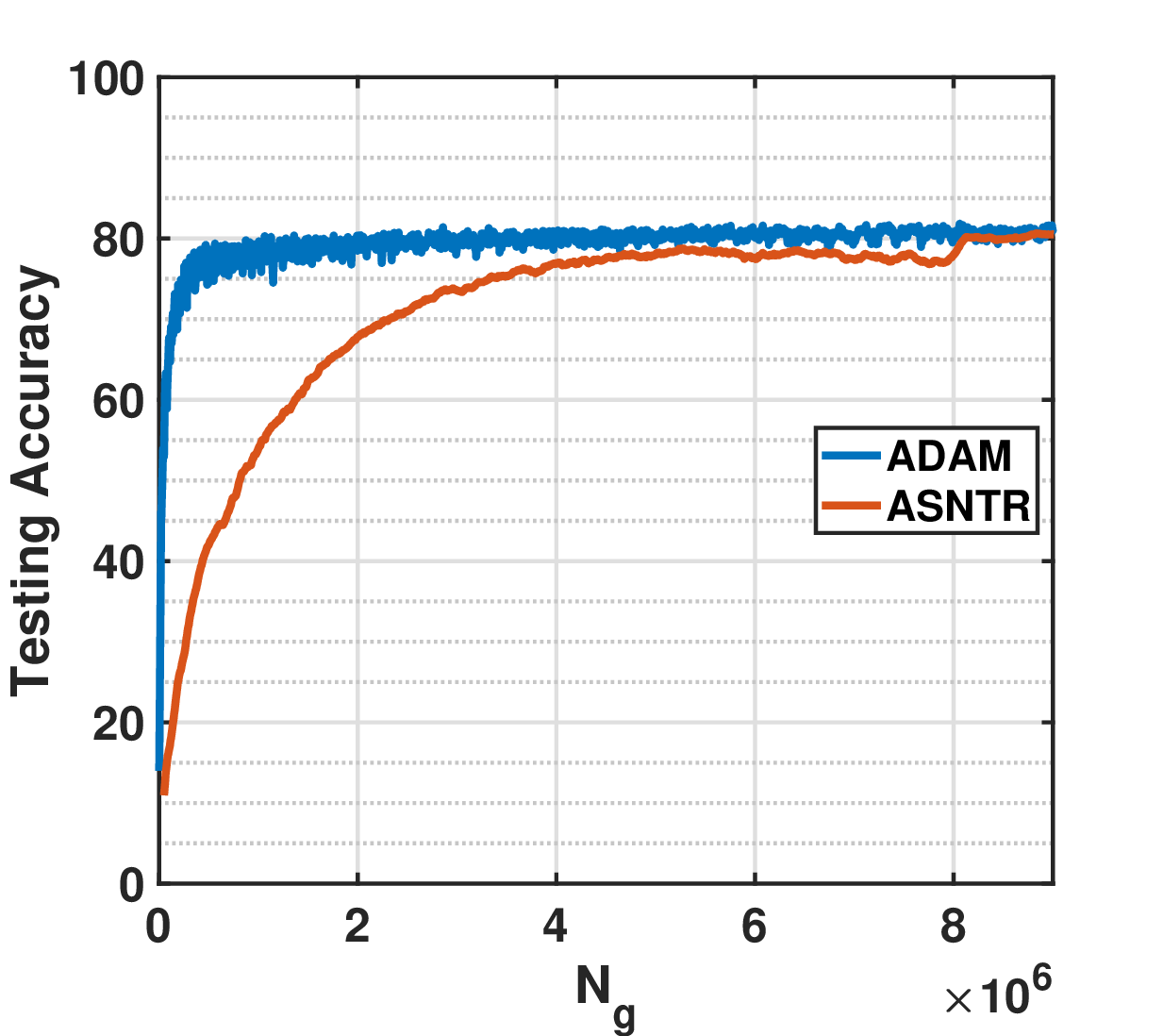}
    \includegraphics[width=5.8cm, height=3.9cm]{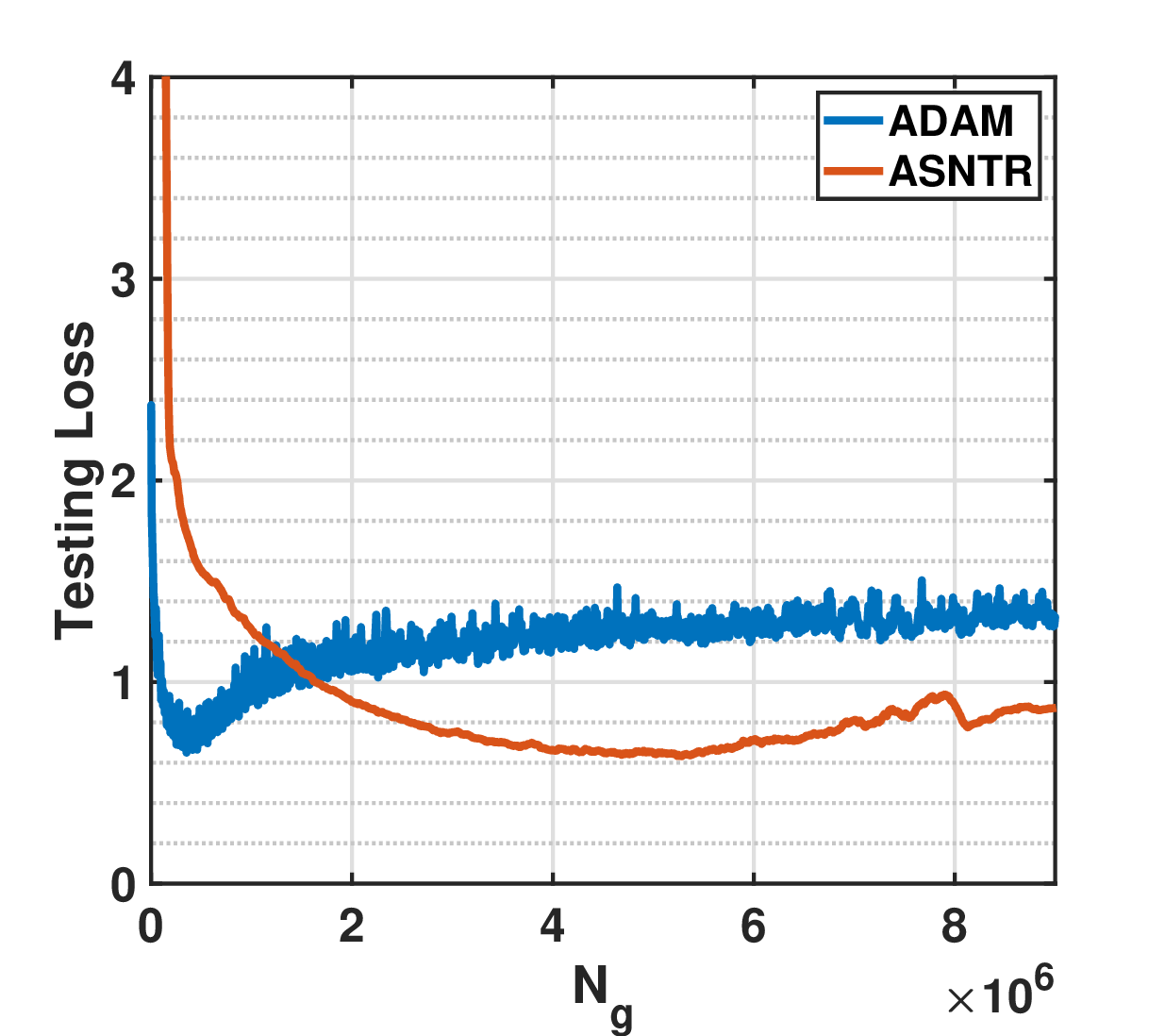}
    \includegraphics[width=5.8cm, height=3.9cm]{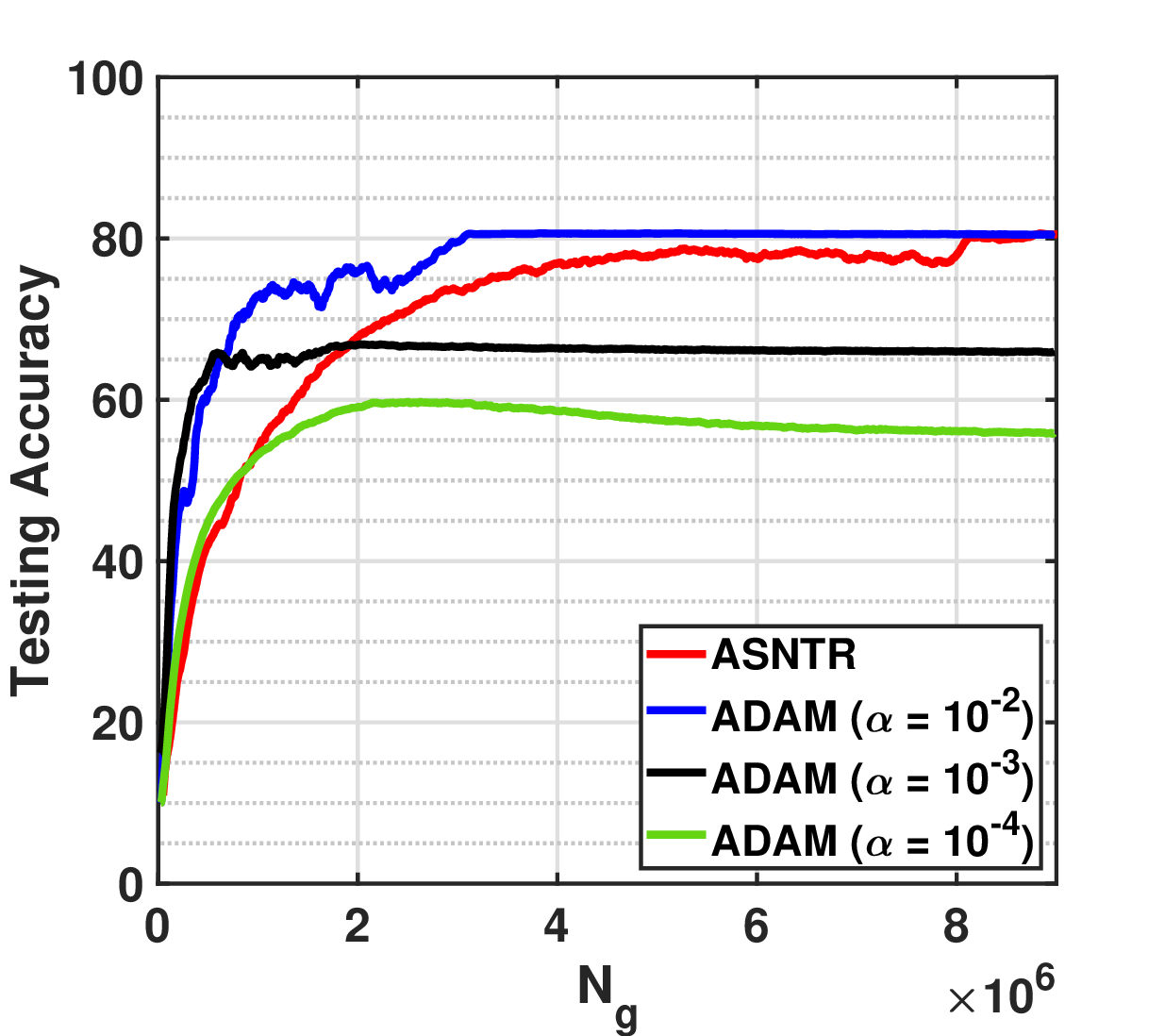}
    \includegraphics[width=5.8cm, height=3.9cm]{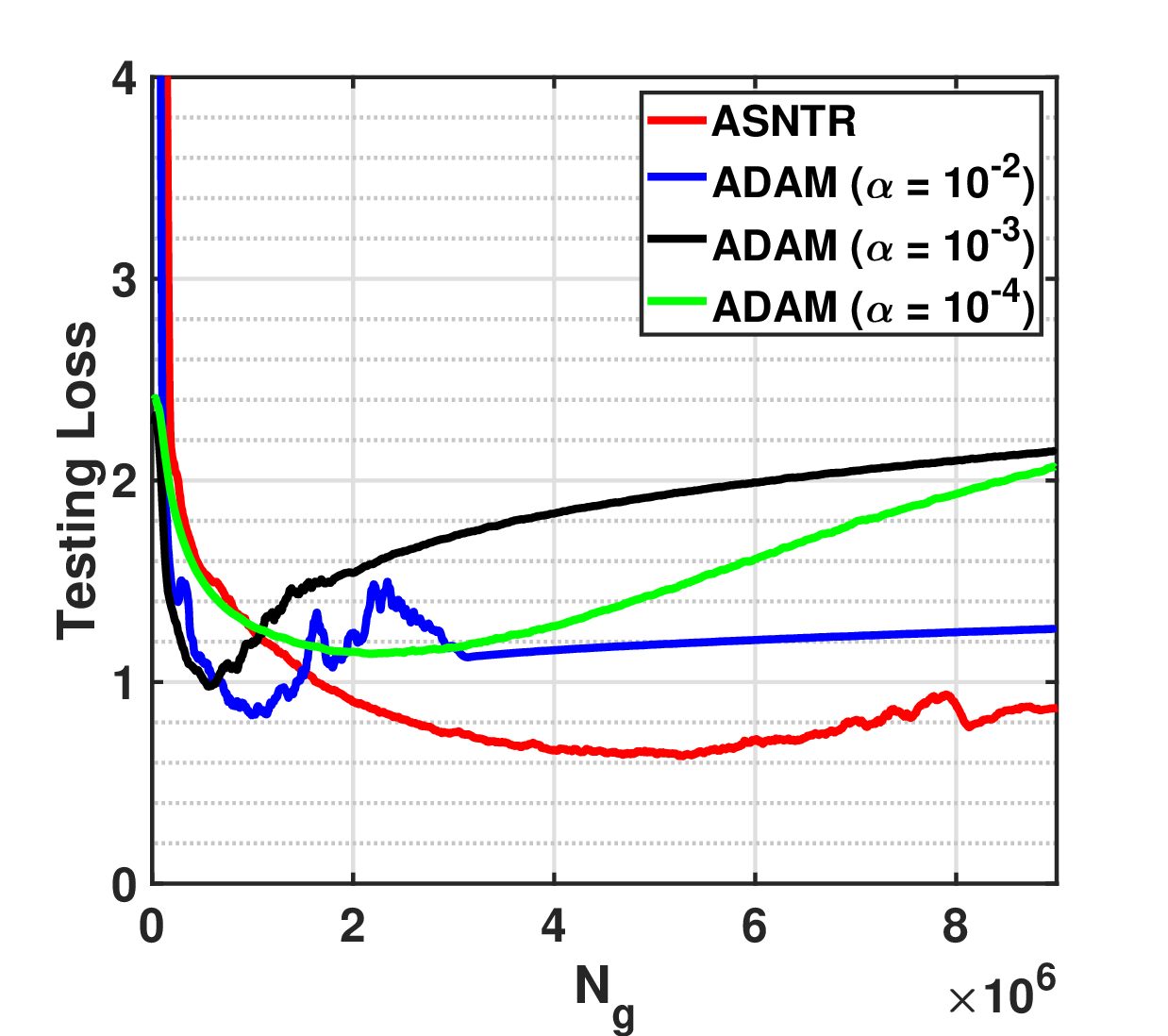}
    \caption{\small Comparison of ASNTR vs. tuned ADAM (first row), and vs. untuned ADAM (second row) on \texttt{CIFAR10} for training \texttt{ResNet-20} with \texttt{rng(42)}.}
    \label{Fig_Acc_Loss_CIFAR}
\end{figure}
\section{Conclusion}\label{Concl}               
In this work, we have presented ASNTR, a second-order non-monotone trust-region method that employs an adaptive subsampling strategy. We have incorporated additional sampling into the TR framework to control the noise and overcome issues in the convergence analysis coming from biased estimators. Depending on the estimated progress of the algorithm, this can yield different scenarios ranging from mini-batch to full sample functions. We provide convergence analysis for all possible scenarios and show that the proposed method achieves almost sure convergence under standard assumptions for the TR framework. The experiments in deep neural network training for both image classification and regression show the efficiency of the proposed method. In comparison to the state-of-the-art second-order method STORM, ASNTR achieves higher testing accuracy with a fixed budget of gradient evaluations. However, our experiments show that the popular first-order method, tuned ADAM using its optimal hyper-parameters, can produce higher accuracy than ASNTR with fixed and 
\begin{figure}[H]
    \centering
    \includegraphics[width=5.8cm, height=3.9cm]{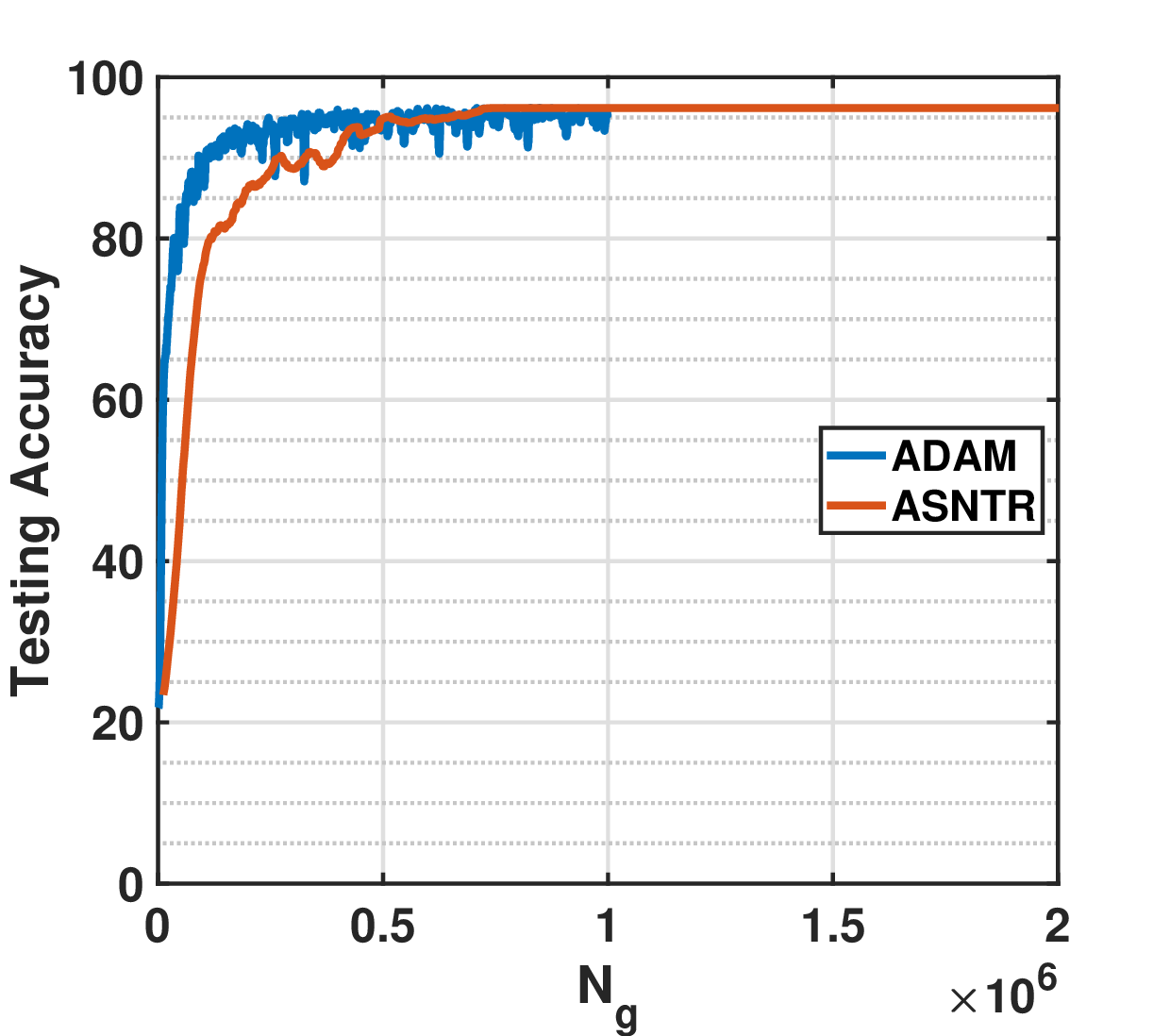}
    \includegraphics[width=5.8cm, height=3.9cm]{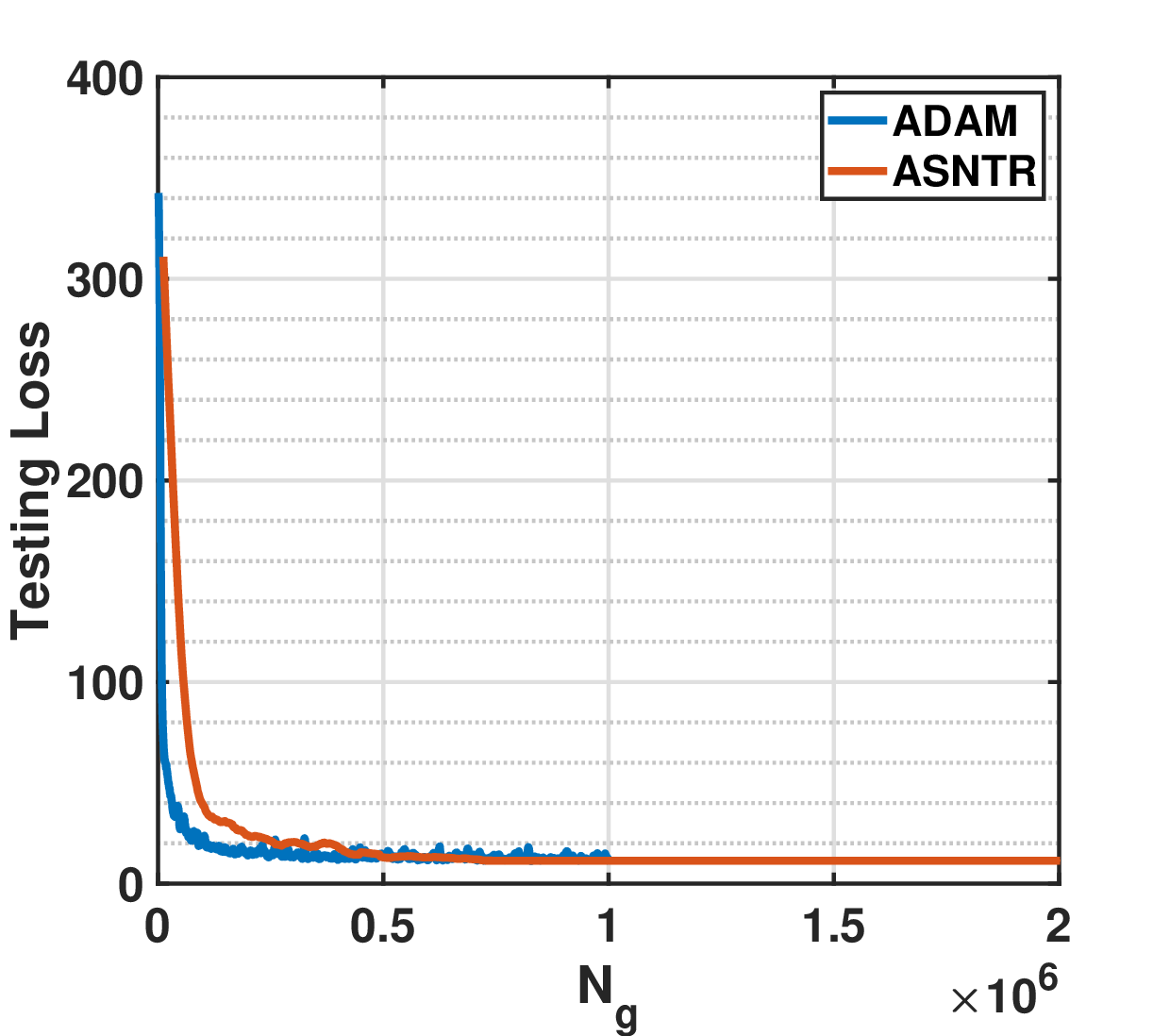}
    \includegraphics[width=5.8cm, height=3.9cm]{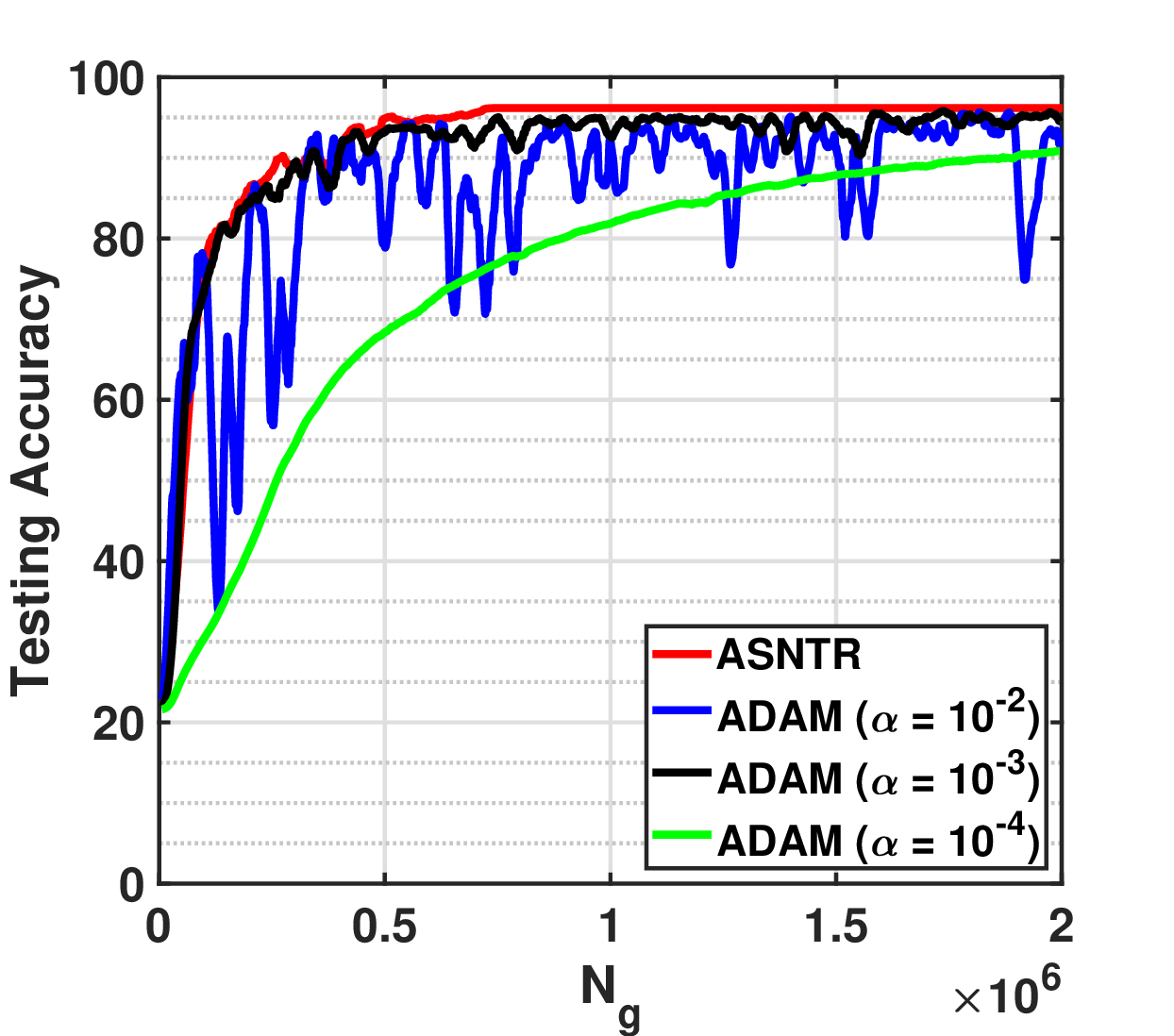}
    \includegraphics[width=5.8cm, height=3.9cm]{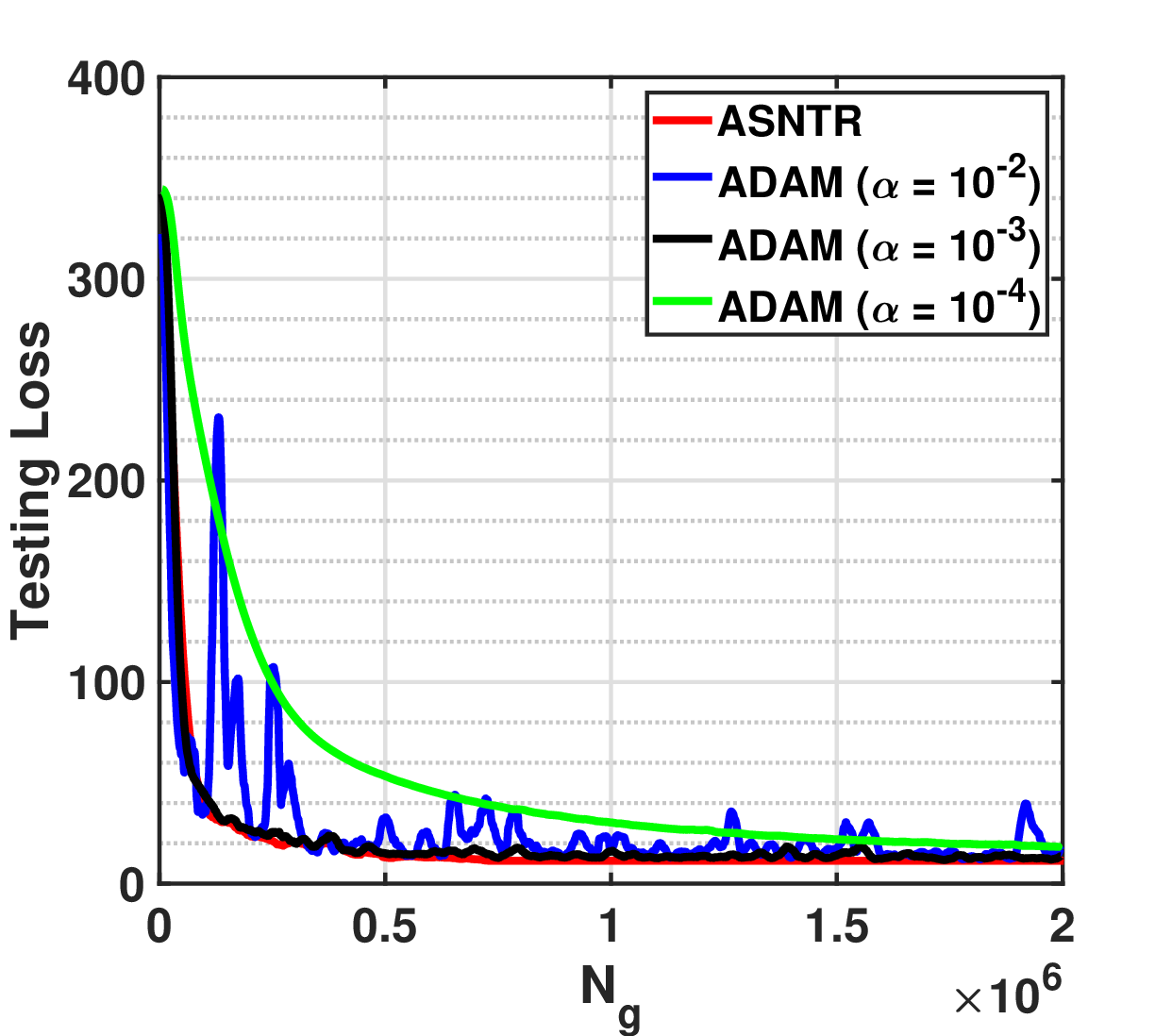}
    \caption{\small Comparison of ASNTR vs. tuned ADAM (first row), and vs. untuned ADAM (second row) on \texttt{DIGITS} for training \texttt{CNN-Rn} with \texttt{rng(42)}.}
    \label{Fig_Acc_Loss_DIGITS}
\end{figure}
\noindent fewer gradient computations if one does not count the effort needed for tuning the parameters. As expected, ASNTR is more robust and performs better than ADAM with suboptimal parameters. Future work on ASNTR could include more specified sample size updates and Hessian approximation strategies.
\bmhead{Acknowledgments}          
We are grateful to the two anonymous referees whose constructive comments helped us to improve the paper. 

The work of N. Kreji\'c and N. Krklec Jerinki\'c was supported by the Science Fund of the Republic of Serbia, Grant no. 7359, Project LASCADO.  A. Mart\'{\i}nez and M. Yousefi gratefully acknowledge the support of the INdAM-GNCS Project CUP$\_$E53C22001930001. The work of A. Mart\'{\i}nez was carried out within the PNRR research activities of the consortium iNEST (Interconnected North-Est Innovation Ecosystem) funded by the European Union Next-GenerationEU (Piano Nazionale di Ripresa e Resilienza (PNRR) – Missione 4 Componente 2, Investimento 1.5 – D.D. 1058 23/06/2022, ECS$\_$00000043). This manuscript reflects only the Authors’ views and opinions, neither the European Union nor the European Commission can be considered responsible for them.
\section*{Competing interests declarations}       
The authors have no relevant financial or non-financial interests to disclose.

\section*{Data availability statements}
The datasets utilized in this research, \texttt{DIGITS}, \texttt{MNIST}, and \texttt{CIFAR-10}, are publicly accessible and commonly employed benchmarks in the field of Machine Learning and Deep Learning, see \url{https://www.mathworks.com/help/deeplearning/ug/data-sets-for-deep-learning.html}, \url{ https://www.kaggle.com/datasets/hojjatk/mnist-dataset} and \cite{krizhevsky2009learning, lecun1998mnist}. 
\bibliography{References}

\end{document}